\documentclass[3p]{elsarticle}
\usepackage{eurosym}
\usepackage{amsmath}
\usepackage{amssymb}
\usepackage{float}
\usepackage{mathrsfs}
\usepackage{graphics}
\usepackage[font=small,labelfont=bf]{caption}
\usepackage{placeins}
\usepackage{float}
\usepackage{varioref}
\usepackage{amssymb,amsmath,mathrsfs}
\usepackage{mathrsfs,enumerate,amssymb,amsmath,color,lineno}
\usepackage{amsfonts}
\usepackage{hyperref}
\usepackage[font=small,labelfont=bf,labelsep=none]{caption}
\newtheorem{theorem}{Theorem}[section]

\newtheorem{corollary}{Corollary}[section]

\newtheorem{definition}{Definition}[section]
\newtheorem{example}{Example}[section]

\newtheorem{lemma}{Lemma}[section]

\newtheorem{remark}{Remark}[section]

\newenvironment{proof}[1][Proof]{\noindent\textbf{#1.} }{\hspace{\stretch{1}}$\square$}
\numberwithin{equation}{section}
\begin{document}
\begin{frontmatter}

\title{\huge Constructing 2-uninorms on bounded lattices by using additive generators\tnoteref{mytitlenote}}
\tnotetext[mytitlenote]{This work is supported by the National Natural Science Foundation of China (No. 12071325).}

\author{Shudi Liang\footnote{\emph{E-mail address}: shudi$\_$liang@163.com}, Xue-ping Wang\footnote{Corresponding author. xpwang1@hotmail.com; fax: +86-28-84761393}\\
\emph{School of Mathematical Sciences, Sichuan Normal University,}\\
\emph{Chengdu 610066, Sichuan, People's Republic of China}}

\begin{abstract}
In this article, we present two methods to construct 2-uninorms on bounded lattices by using additive generators, which are further used for inducing uninorms, nullnorms, uni-nullnorms and null-uninorms, respectively. We also provide some examples for illustrating the constructing methods of 2-uninorms.
\end{abstract}

\end{frontmatter}

\section{Introduction}
Triangular norms (t-norms for short) and triangular conorms (t-conorms for short) are associative, commutative and monotone binary
operations with the neutral element 1 and 0, respectively.
Schweizer and Sklar \cite{SS1958} studied t-norms and t-conorms on $[0,1]$
based on Menger's notion \cite{M1942}.
T-norms and t-conorms have been proved to be useful in many fields
like fuzzy set theory \cite{Klement}, fuzzy logic \cite{BJ2008}, fuzzy systems
modeling \cite{Y2001}, and probabilistic metric spaces \cite{SS1961,SS1983}.

As important generalizations of t-norms and t-conorms,
the concepts of uninorms and nullnorms were introduced on the unit interval $[0, 1]$ by Yager and Rybalov \cite{YR1996} and Calvo et al. \cite{CB2001}, respectively. Uninorms allow for a neutral element anywhere in the unit interval, whereas nullnorms allow for a zero element $k$ anywhere in the unit interval, while keeping 0 as neutral element on $[0, k]$ and 1 as neutral element on $[k, 1]$. A series of works have been done for uninorms \cite{C2020u,DH2020,J2021,KEM2017} and nullnorms \cite{C2020n,C2020,SL2020r,WLC2022,ZWOB2022}, respectively.

To unify uninorms and nullnorms, 2-uninorms were first investigated by Akella \cite{A2007}. Since then, some properties of 2-uninorms on the unit interval are studied. For example, Dryga\'{s} and Rak \cite{DR2016} solved the functional equations of distributivity between 2-uninorms.
Wang and Qin \cite{WQ2017} studied the distributivity equations for 2-uninorms over semi-uninorms. Zong et al. \cite{ZSLB2018} described the structures of 2-uninorms. Then Sun and Liu \cite{SL2020} investigated the left (resp. right) distributivity of semi-t-operators over 2-uninorms. Almost at the same time, Zhang and Qin \cite{ZQ2020} obtained some sufficient and necessary conditions of the distributivity equations between five classes of basic 2-uninorms and overlap (resp. grouping) functions. Wang et al. \cite{WZS2021} introduced the discrete 2-uninorms. Huang and Qin \cite{HQ2021} made a deep study on the migrativity of uninorms over 2-uninorms.

As a bounded lattice is more general than $[0,1]$, the study of 2-uninorms on the unit interval has already been extended to bounded lattices. For instance, Ertu\v{g}rul \cite{E2017} and Xie and Yi \cite{XY2022} gave the constructions of 2-uninorms. Recently, Sun and Liu \cite{SL2021} explored the additive generators of t-norms and t-conorms. Also, He and Wang \cite{HW2021}
studied the additive generators of uninorms, and they even extended the classical additive generators to partially ordered cases by adding some conditions. As 2-uninorms are generalizations of uninorms,
this leads to a natural question: Could we construct 2-uninorms on bounded lattices by using additive generators? This article will focus on this question.

The remainder of this article is organized as follows. In Section \ref{sec2}, we provide the necessary background material. Section \ref{sec3} is devoted to the constructions of 2-uninorms on bounded lattices based on additive generators. Finally, conclusions are drawn in Section \ref{sec6}.

\section{Preliminaries}\label{sec2}

This section presents some basic definitions and results that are used latter.

A {\it lattice} \cite{GB1997} is a nonempty set $L$ equipped with a partial order $\leq$ such that
any two elements $x$ and $y$ have a greatest lower bound (called meet or infimum), denoted by
$x \wedge y$, as well as a smallest upper bound (called join or supremum), denoted by $x \vee y$.
A lattice is called {\it bounded} if it has a top element $1_L$ and a bottom element $0_L$.
For short, we use the notation $L$ instead of a bounded lattice $(L, \leq, 0_L, 1_L)$ throughout this article. Let $x,y\in L$. The elements $x$ and $y$ are called {\it comparable}
if $x \leq y$ or $y \leq x$. Otherwise, $x$ and $y$ are called {\it incomparable}. The symbol $x \parallel y$ is used when $x$ and $y$ are incomparable.
If $x$ and $y$ are comparable, then we use the symbol $x\nparallel y$.
In what follows, $I_a$ denotes the set of all incomparable elements with $a\in L$, i.e., $I_a=\{x\in L: x\| a\}$.
Let $a,b\in L$ with $a \leq b$. An interval $[a,b]$
is defined as $[a,b] = \{x \in L \mid a \leq x \leq b\}$, other intervals can be defined similarly,
$(a,b] =\{x \in L \mid a < x \leq b\}$, $[a,b) =\{x \in L \mid a \leq x < b \}$, $(a,b) =\{x \in L \mid a < x < b \}$.

\begin{definition}[\cite{BM1999,MW1991}]\label{def1}
\emph{
\begin{enumerate}[{\rm (i)}]
\item A binary operation $T: L^2 \to L$ is called a {\it t-norm} if it is commutative, associative, and increasing with respect to both variables and it satisfies $T(x,1_L)=x$ for all $x\in L$.
\item A binary operation $S: L^2 \to L$ is called a {\it t-conorm} if it is commutative, associative, and increasing with respect to both variables and it satisfies $S(x,0_L)=x$ for all $x\in L$.
\end{enumerate}
}
\end{definition}

\begin{definition}[\protect\cite{KM2015,YR1996}]\label{def2.2}
\emph{A binary operation $U: L^2 \to L$ is called a {\it uninorm} if it has commutativity, associativity, and increasing with respect to both variables and a neutral element $e\in L$.
}
\end{definition}

Obviously, a t-norm (resp. t-conorm) on $L$ is exactly a
uninorm with the neutral element $e=1_L$ (resp. $e=0_L$).

\begin{definition}[\protect\cite{CB2001,KIM2015}]\label{def2.3}
\emph{A binary operation $%
V:L^{2}\rightarrow L$ is called a {\it nullnorm} if it is
commutative, associative, increasing with respect to both variables, and
there exists an element $a\in L$, which is called a {\it{zero element}}
for $V$, such that $V(x,0_L)=x$ for all $x\in [0_L,a]$ and $V(x,1_L)=x$ for
all $x\in [a,1_L]$.}
\end{definition}

Obviously, a t-norm (resp. t-conorm) on $L$ is exactly a
nullnorm with the zero element $a=0_L$ (resp. $a=1_L$).

\begin{definition}[\cite{SL2021}]
\emph{Let $P,Q$ be two partially ordered sets and $f:P \to Q$ be a non-constant injective monotone function. If
$f$ is increasing, then a pseudo-inverse $f^{(-1)}: Q\to P$ is given by \eqref{def-1} when it exists.
\begin{equation}\label{def-1}
f^{(-1)}(y)=
\begin{cases}
\inf\{x\in P|f(x)>y\}, & \mbox{ if card}\{f^{-1}(y)\}=0, \\
f^{-1}(y), & \mbox{ if card}\{f^{-1}(y)\}=1.\\
\end{cases}
\end{equation}
If $f$ is decreasing then a pseudo-inverse $f^{(-1)} : Q\to P$ is given by \eqref{def-2} when it exists.
\begin{equation}\label{def-2}
f^{(-1)}(y)=
\begin{cases}
\sup\{x\in P|f(x)>y\}, & \mbox{ if card}\{f^{-1}(y)\}=0, \\
f^{-1}(y), & \mbox{ if card}\{f^{-1}(y)\}=1.\\
\end{cases}
\end{equation}
}
\end{definition}

\begin{definition}[\cite{HW2021}]
\emph{Let $0\in A\subseteq [-\infty, +\infty]$. For two non-zero elements $x,a\in A$, if there exists $b\in A$ fulfilling $x=a+b$ and
$ab>0$, then we call $a$ a {\it summand} of $x$ in $A$.}
\end{definition}

\begin{remark}[\cite{HW2021}]\label{ex-1}
\emph{Let $0\in A\subseteq [-\infty, +\infty]$ and $x\in A$. If $x=0$ then $x$ has no summands. If $x\neq 0$, then $x$ is always a summand of itself
and each summand $a$ of $x$ satisfies $xa>0$}
\end{remark}
\begin{definition}[\protect\cite{SWQ2017,WZL2020}]\label{def2.6}
\emph{ A binary operation $F : L^2\to L$ is called a {\it uni-nullnorm} if it satisfies the commutativity, associativity, monotonicity with respect to both variables, and there exist some elements $e, a\in L$ with $0_L\leq e < a\leq 1_L$ such that
$F(e, x) = x$ for all $x\in [0_L, a]$ and $F(x, 1_L) = x$ for all $x\in [a, 1_L]$.}
\end{definition}
\begin{definition}[\protect\cite{A2007,E2017}]\label{def2.7}
\emph{Let $e_1, a, e_2\in L$ with $0_L\leq e_1\leq a\leq e_2\leq 1_L$. A binary operation $\mathbb{U}: L^2\to L$ is called a {\it 2-uninorm} if it is commutative, associative,
increasing with respect to both variables and fulfilling $\mathbb{U}(e_1,x)=x$ for all $x\leq a$ and $\mathbb{U}(e_2,x)=x$ for all $x\geq a$.} \end{definition}

From Definitions \ref{def2.2}, \ref{def2.3}, \ref{def2.6} and \ref{def2.7}, one can easily check the following remark.
\begin{remark}\label{remar2.2}
\emph{
\begin{enumerate}[{\rm (i)}]
\item
A 2-uninorm with $e_2=1_L$ is a uni-nullnorm.
\item
A 2-uninorm with $e_1=0_L$ is a null-uninorm.
\item
A 2-uninorm with $a=1_L$ or $a=0_L$ is a uninorm.
\item
A 2-uninorm with $e_1=0_L$ and $e_2=1_L$ is a nullnorm.
\end{enumerate}
}
\end{remark}

\section{Constructions of 2-uninorms on bounded lattices}\label{sec3}
In this section, we first give two new methods for constructing 2-uninorms on bounded lattices $L$, which are further used for inducing uninorms, nullnorms, uni-nullnorms and null-uninorms, respectively. We then provide some examples to illustrate the new methods.

Let $a\in L$. Denoted by $\mathcal{F}_a=\{f(x)|0< f(x)< f(a) \mbox{ and } x\in I_a\}$ when $f:L\to [-\infty,+\infty]$ is injective increasing, and by $\mathcal{G}_a=\{f(x)|f(a)< f(x)< 0 \mbox{ and } x\in I_a\}$ when $f:L\to [-\infty,+\infty]$ is injective decreasing.

\begin{lemma}\label{lem-1}
Let $f:L\to [-\infty,+\infty]$ be an injective increasing function, and $x,y\in L$ with $x\leq y$. If $f(x)\leq f(y)\leq f(a)$ and $x\in I_a$, then $y\in I_a$.
\end{lemma}
\begin{proof}
Assume that $y\notin I_a$. Then $y\leq a$ since $f(y)\leq f(a)$ and $f$ is injective increasing. Thus from $x\leq y$ we obtain $x\leq y\leq a$,  contrary to the fact that $x\in I_a$.
\end{proof}

\begin{theorem}\label{thm-1}
Let $e_1, a, e_2\in L$ with $0_L\leq e_1\leq a\leq e_2\leq 1_L$, and $f:L\to [-\infty,+\infty]$ be an injective increasing function with $f(e_1)=0$.
If $f$ satisfies the following five conditions: for all $x,y\in L$,
\begin{enumerate}[{\rm (i)}]
  \item
if $(f(x),f(y))\in [0,f(a)]^2\cup [-\infty, 0]^2$ then
$$\min\{f(x)+f(y),f(a)\}\in Ran(f)\cup [-\infty,f(0_L)),$$
  \item
 for all $(f(x),f(y))\in [0,f(a)]^2\cup [-\infty, 0]^2$, if $f(x)$ and $f(y)$ have at least one same summand $f(z)\in Ran(f)$ then $x\nparallel y$,
\item
if $f(x)<0< f(y)\leq f(a)$ and $x<e_1$, then $x\nparallel y$,
  \item
if $f(a)\leq f(x)< f(y)$ or $f(x)<0<f(a)\leq f(y)$, then $x\nparallel y$,
  \item
for all $f(x),f(y)\in (0,f(a))$, if $f(x)+f(y)\in Ran(f)$ and $0< f(x)+f(y)\leq f(a)$ then $f^{-1}(f(x)+f(y))\nparallel a$,
\end{enumerate}
then the following function $\mathbb{U}: L^2\to L$ is a 2-uninorm, and we call $f$ an additive generator of $\mathbb{U}$. For all $x,y\in L$,
\begin{equation}\label{eq-1}
\mathbb{U}(x, y)=
\begin{cases}
f^{(-1)}(f(x)+f(y)), & (f(x),f(y))\in [-\infty, 0]^2, \\
f^{(-1)}(\min\{f(x)+f(y), f(a)\}), & (f(x),f(y))\in [0,f(a)]^2 \mbox{ and } x\notin I_a \mbox{ and } y\notin I_a, \\
a, & (f(x),f(y))\in (f(a),+\infty]\times [0,f(a)]\\
&~~~~\cup [0,f(a)]\times (f(a),+\infty]\cup [0,f(a)]\times \mathcal{F}_a\\
&~~~~\cup \mathcal{F}_a\times [0,f(a)],\\
f^{(-1)}(\max\{f(x),f(y)\}), & (f(x),f(y))\in (f(e_2),+\infty]^2\cup (f(e_2),+\infty]\\
&~~~~\times (f(a),f(e_2)]\cup (f(a),f(e_2)]\times (f(e_2),+\infty],\\
f^{(-1)}(\min\{f(x),f(y)\}), & otherwise.
\end{cases}
\end{equation}
\end{theorem}

\begin{proof}
First it is necessary to check that $\mathbb{U}$ is well-defined. The proof is made in five cases:
\begin{enumerate}[{\rm (I)}]
  \item
If $(f(x),f(y))\in [f(a),+\infty)\times [0,f(a)]\cup [0,f(a)]\times [f(a),+\infty)\cup [0,f(a)]\times \mathcal{F}_a\cup \mathcal{F}_a\times [0,f(a)]$, then $\mathbb{U}(x,y)=a$. Thus $\mathbb{U}$ is well-defined by \eqref{eq-1}.
  \item
If $(f(x),f(y))\in [-\infty,0]^2$, then $f(x)+f(y)\leq f(a)$, subsequently,
$$f^{(-1)}(\min\{f(x)+f(y),f(a)\})=f^{(-1)}(f(x)+f(y)).$$
There are two subcases as follows.

Subcase (II-a). If $f(x)+f(y)\in Ran(f)$ then clearly card$\{f^{-1}(f(x)+f(y))\}=1$ since $f$ is injective. Thus
$\mathbb{U}(x,y)$ is well-defined by \eqref{def-1}.

Subcase (II-b). If $f(x)+f(y)\in [-\infty,f(0_L))$ then card$\{f^{-1}(f(x)+f(y))\}=0$. By \eqref{eq-1}, we know that
 $$f^{(-1)}(f(x)+f(y))=\inf\{z\in L:f(z)>f(x)+f(y)\}=\wedge L=0_L.$$
 Hence $\mathbb{U}(x,y)$ is well-defined by \eqref{def-1}.
 \item
If $(f(x),f(y))\in [0,f(a)]^2$, $x,y\notin I_a$ and $f(x)+f(y)\leq f(a)$, then $f(x)+f(y)\in Ran(f)$, subsequently,
$$f^{(-1)}(\min\{f(x)+f(y),f(a)\})=f^{(-1)}(f(x)+f(y)).$$
Thus, similar to (II-a), $\mathbb{U}(x,y)$ is well-defined.
 \item
If $(f(x),f(y))\in [0,f(a)]^2$, $x,y\notin I_a$ and $f(x)+f(y)\geq f(a)$, then
$$f^{(-1)}(\min\{f(x)+f(y),f(a)\})=f^{(-1)}(f(a))=a.$$

Thus $\mathbb{U}$ is well-defined by \eqref{eq-1}.
  \item
For the other cases, without loss of generality, we suppose that $f(x)\leq f(y)$. Then
$$f^{(-1)}(\min\{f(x),f(y)\})=f^{(-1)}(f(x))=x,$$
and
$$f^{(-1)}(\max\{f(x),f(y)\})=f^{(-1)}(f(y))=y.$$
\end{enumerate}

Thus $\mathbb{U}(x,y)$ is well-defined by \eqref{eq-1}.

Therefore, from (I)-(V), we know that $\mathbb{U}(x,y)$ is well-defined for any $x,y\in L$.

Next, we prove that $\mathbb{U}$ is a 2-uninorm. First, from \eqref{eq-1}, we have the following three statements.
\begin{enumerate}[{\rm (a)}]
  \item
$\mathbb{U}$ is a commutative binary operation on $L$.
  \item
  If $x\in [0_L,a]$, i.e., $f(x)\leq f(a)$ and $x\notin I_a$, then $\mathbb{U}(e_1,x)=f^{(-1)}(\min\{f(e_1)+ f(x), f(a)\})=f^{(-1)}((f(x))=x$ since $f$ is injective increasing.
  \item
  If $x\in [a,1_L]$, i.e., $f(x)\geq f(a)$, then $\mathbb{U}(e_2,x)=f^{(-1)}(\min\{f(e_2), f(x)\})=f^{(-1)}(f(x))=x$ for $f(a)\leq f(x)\leq f(e_2)$ and
  $\mathbb{U}(e_2,x)=f^{(-1)}(\max\{f(e_2), f(x)\})=f^{(-1)}(f(x))=x$ for $f(x)\geq f(e_2)$ since $f$ is injective increasing.
\end{enumerate}
Then it remains to show the monotonicity and associativity of the binary operation $\mathbb{U}$.
\begin{enumerate}[{\rm (1)}]
\item
\textbf{Monotonicity. }Let $x,y,z\in L$ with $x\leq y$. We need to verify $\mathbb{U}(x,z)\leq \mathbb{U}(y,z).$  Noticing that $f(x)\leq f(y)$ since $x\leq y$ and $f$ is injective increasing. Therefore, the proof is split into all possible cases as follows.

1. $f(z)=0$.

1.1 If $f(x)\leq f(y)< 0$, or $f(x)< 0\leq f(y)\leq f(a)$ and $y\notin I_a$, then

$$\mathbb{U}(x,z)=f^{(-1)}(f(x)+f(z))=x\leq y=f^{(-1)}(f(y)+f(z))=\mathbb{U}(y,z).$$

1.2 If $f(x)< 0\leq f(y)\leq f(a)$ and $y\in I_a$, or $f(x)< 0<f(a)<f(y)$, then from (iv), we have

$$\mathbb{U}(x,z)=f^{(-1)}(f(x)+f(z))=f^{(-1)}(f(x))=x\leq a=\mathbb{U}(y,z).$$

1.3 $0\leq f(x)\leq f(y)\leq f(a)$.

1.3.1 If $x,y\notin I_a$, then $\mathbb{U}(x,z)=x\leq y=\mathbb{U}(y,z)$.

1.3.2 If $x\in I_a$, then $y\in I_a$ by Lemma \ref{lem-1}, subsequently, $\mathbb{U}(x,z)=a=\mathbb{U}(y,z)$.

1.3.3 If $x\notin I_a$ and $y\in I_a$, then $\mathbb{U}(x,z)=x\leq a=\mathbb{U}(y,z)$.

1.4 If $0\leq f(x)\leq f(a)<f(y)$ and $x\notin I_a$, then $\mathbb{U}(x,z)=f^{(-1)}(f(x)+f(z))=x\leq a=\mathbb{U}(y,z)$.

1.5 If $0\leq f(x)\leq f(a)<f(y)$ and $x\in I_a$, then $\mathbb{U}(x,z)=a=\mathbb{U}(y,z)$.

1.6 If $f(a)<f(x)\leq f(y)$, then $\mathbb{U}(x,z)=a=\mathbb{U}(y,z)$.

2. $f(z)<0$.

2.1. If $f(x)\leq f(y)\leq 0$, then we have $f(x)+f(z)\leq 0< f(a)$, subsequently,

$$\mathbb{U}(x,z)=f^{(-1)}(\min\{f(x)+f(z),f(a)\})=f^{(-1)}(f(x)+f(z)).$$

Analogously, we obtain $\mathbb{U}(y,z)=f^{(-1)}(\min\{f(y)+f(z),f(a)\})=f^{(-1)}(f(y)+f(z))$.

2.1.1. If $f(x)+f(z)\in [-\infty,f(0_L))$, then $\mathbb{U}(x,z)=f^{(-1)}(f(x)+f(z))=\wedge L=0_L\leq \mathbb{U}(y,z)$.

2.1.2. If $f(x)+f(z)\in Ran(f)$, then $f(y)+f(z)\in Ran(f)$. Noticing that $f(z)$ is the same summand of both $f(x)+f(z)$ and $f(y)+f(z)$, thus $\mathbb{U}(x,z)\nparallel \mathbb{U}(y,z)$ by (ii).
Therefore, $\mathbb{U}(x,z)\leq \mathbb{U}(y,z)$ since $f(x)+f(z)\leq f(y)+f(z)$ and $f$ is injective increasing.

2.2. If $f(x)< 0< f(y)$, then $\mathbb{U}(x,z)=f^{(-1)}(f(x)+f(z))$, and

$$\mathbb{U}(y,z)=f^{(-1)}(\min\{f(x),f(z)\})=f^{(-1)}(f(z))=z.$$ \label{eq2.2}

2.2.1. If $f(x)+f(z)\in [-\infty,f(0_L))$, then $\mathbb{U}(x,z)=0_L\leq \mathbb{U}(y,z)$ from 2.1.1.

2.2.2. If $f(x)+f(z)\in Ran(f)$, then $f(z)$ is the same summand of both $f(x)+f(z)$ and $f(z)$ by Remark \ref{ex-1}. Thus $\mathbb{U}(x,z)\nparallel f^{-1}(f(z))$ by (ii). Therefore, from $f(x)\leq 0$ we have $\mathbb{U}(x,z)\leq f^{-1}(f(z))=z=\mathbb{U}(y,z)$ since $f(x)+f(z)\leq f(z)$ and $f$ is injective increasing.

2.3. If $0\leq f(x)\leq f(y)$, then $\mathbb{U}(x,z)=f^{(-1)}(\min\{f(x),f(z)\})=f^{(-1)}(f(z))=z$ and

$$\mathbb{U}(y,z)=f^{(-1)}(\min\{f(y),f(z)\})=f^{(-1)}(f(z))=z,$$

therefore, $\mathbb{U}(x,z)=z=\mathbb{U}(y,z)$.

3. $f(z)> 0$.

3.1 If $f(x)\leq f(y)<0$, then

$$\mathbb{U}(x,z)=f^{(-1)}(\min\{f(x),f(z)\})=x\leq y=f^{(-1)}(\min\{f(y),f(z)\})=\mathbb{U}(y,z).$$

3.2 $f(x)<0=f(y)$.

3.2.1 If $0<f(z)\leq f(a)$ and $z\notin I_a$, then from (iii) we have

$$\mathbb{U}(x,z)=f^{(-1)}(\min\{f(x),f(z)\})=x\leq z=f^{(-1)}(\min\{f(y)+f(z),f(a)\})=\mathbb{U}(y,z).$$

3.2.2 If $0<f(z)\leq f(a)$ and $z\in I_a$, then from (iii) we have

$$\mathbb{U}(x,z)=f^{(-1)}(\min\{f(x),f(z)\})=x\leq a=\mathbb{U}(y,z).$$

3.2.3 If $f(z)> f(a)$, then from (iii), we have

$$\mathbb{U}(x,z)=f^{(-1)}(\min\{f(x),f(z)\})=x\leq a=\mathbb{U}(y,z).$$

3.3 $f(x)=0<f(y)\leq f(a)$.

3.3.1 If $0< f(z)\leq f(a)$, and $z\in I_a$, then $\mathbb{U}(x,z)= a=\mathbb{U}(y,z)$.

3.3.2 $0< f(z)\leq f(a)$ and $y,z\notin I_a$.

3.3.2.1 If $f(y)+f(z)\leq f(a)$, then from (ii), we have

$$\mathbb{U}(y,z)=f^{(-1)}(\min\{f(y)+f(z),f(a)\})=f^{(-1)}(f(y)+f(z))\geq f^{(-1)}(z)=z=\mathbb{U}(x,z).$$

3.3.2.2 If $f(y)+f(z)\geq f(a)$, then from $z\notin I_a$ and $f(z)\leq f(a)$, we have

$$\mathbb{U}(y,z)=f^{(-1)}(\min\{f(y)+f(z),f(a)\})=a\geq z=f^{(-1)}(\min\{f(x)+f(z),f(a)\})=\mathbb{U}(x,z).$$

3.3.3 If $0< f(z)\leq f(a)$, $y\in I_a$ and $z\notin I_a$, then

$$\mathbb{U}(y,z)=a\geq z=f^{(-1)}(\min\{f(x)+f(z),f(a)\})=\mathbb{U}(x,z).$$

3.3.4 If $f(z)>f(a)$, then $\mathbb{U}(y,z)=a=\mathbb{U}(x,z)$.

3.4 If $f(x)< 0< f(y)$, then

\begin{eqnarray}\label{eq-2.2}
\mathbb{U}(x,z)=f^{(-1)}(\min\{f(x),f(z)\})=x.
\end{eqnarray}

3.4.1 $(f(y),f(z))\in (0,f(a)]^2$.

3.4.1.1 $y,z\notin I_a$.

3.4.1.1.1 If $f(y)+f(z)\leq f(a)$ then, because of \eqref{eq-2.2} and (ii), we have

$$\mathbb{U}(y,z)=f^{(-1)}(\min\{f(y)+f(z),f(a)\})=f^{(-1)}(f(y)+f(z))\geq f^{(-1)}(y)=y\geq x=\mathbb{U}(x,z).$$

3.4.1.1.2 If $f(y)+f(z)\geq f(a)$ then, because of (iv), we have

$$\mathbb{U}(y,z)=f^{(-1)}(\min\{f(y)+f(z),f(a)\})=f^{(-1)}(f(a))=a\geq x=\mathbb{U}(x,z).$$

3.4.1.2 If $z\in I_a$ or $y\in I_a$, then from \eqref{eq-2.2} and (iv) we have $\mathbb{U}(y,z)=a\geq x=\mathbb{U}(x,z)$.

3.4.2 $(f(y),f(z))\in (f(a),f(e_2)]^2$.

3.4.2.1 If $f(y)\leq f(z)$, then

$$\mathbb{U}(y,z)=f^{(-1)}(\min\{f(y),f(z)\})= f^{(-1)}(f(y))=y\geq x= \mathbb{U}(x,z).$$

3.4.2.2 If $f(y)\geq f(z)$, then from (iv) we have

$$\mathbb{U}(y,z)=f^{(-1)}(\min\{f(y),f(z)\})= f^{(-1)}(f(z))=z\geq x= \mathbb{U}(x,z).$$

3.4.3 $(f(y),f(z))\in (f(e_2),+\infty]^2\cup (f(e_2),+\infty]\times (f(a),f(e_2)]\cup (f(a),f(e_2)]\times(f(e_2),+\infty]$.

3.4.3.1 If $f(y)\geq f(z)$, then

$$\mathbb{U}(y,z)=f^{(-1)}(\max\{f(y),f(z)\})= f^{(-1)}(y)=y\geq x=\mathbb{U}(x,z).$$

3.4.3.2 If $f(y)\leq f(z)$, then from (iv), we have

$$\mathbb{U}(y,z)=f^{(-1)}(\max\{f(y),f(z)\})= f^{(-1)}(z)=z\geq x =\mathbb{U}(x,z).$$

3.4.4 If $(f(y),f(z))\in [f(a),+\infty]\times [0,f(a)]\cup [0,f(a)]\times [f(a),+\infty]$, then from (iv), we have

$$\mathbb{U}(y,z)=a\geq x=\mathbb{U}(x,z).$$

3.5 $0< f(x)\leq f(y)\leq f(a)$.

3.5.1 $0< f(z)\leq f(a)$.

3.5.1.1 If $z\in I_a$, then $\mathbb{U}(x,z)=a=\mathbb{U}(y,z)$.

3.5.1.2 $x,y,z\notin I_a$.

3.5.1.2.1 If $f(a)\leq f(x)+f(z)\leq f(y)+f(z)$, then $\mathbb{U}(x,z)=a=\mathbb{U}(y,z)$.

3.5.1.2.2 If $f(x)+f(z)\leq f(a)\leq f(y)+f(z)$ then, because of (v), we have

$$\mathbb{U}(x,z)=f^{(-1)}(\min\{f(x)+f(z),f(a)\})=f^{(-1)}(f(x)+f(z))\leq a=\mathbb{U}(y,z).$$

3.5.1.2.3 If $f(x)+f(z)\leq f(y)+f(z)\leq f(a)$ then, from (ii), we have

$$\mathbb{U}(x,z)=f^{(-1)}(f(x)+f(z))\leq f^{(-1)}(f(y)+f(z))=\mathbb{U}(y,z).$$

3.5.1.3 If $z\notin I_a$ and $x\in I_a$, then $y\in I_a$ by Lemma \ref{lem-1}, subsequently, $\mathbb{U}(x,z)=a=\mathbb{U}(y,z)$.

3.5.1.4 If $z,x\notin I_a$ and $y\in I_a$, then from (v), we have $\mathbb{U}(x,z)=f^{(-1)}(f(x)+f(z))\leq a=\mathbb{U}(y,z)$.

3.5.2 If $f(z)> f(a)$, then $\mathbb{U}(x,z)=a=\mathbb{U}(y,z)$.

3.6 $0\leq f(x)\leq f(a)< f(y)$.

3.6.1 $0< f(z)\leq f(a)$.

3.6.1.1 $x,z\notin I_a$.

3.6.1.1.1 If $f(x)\neq 0$ and $f(x)+f(z)\leq f(a)$ then, because of (v), we have

$$\mathbb{U}(x,z)=f^{(-1)}(\min\{f(x)+f(z),f(a)\})=f^{(-1)}(f(x)+f(z))\leq a=\mathbb{U}(y,z).$$

3.6.1.1.2 If $f(x)=0$, then $\mathbb{U}(x,z)=z\leq a=\mathbb{U}(y,z)$.

3.6.1.1.3 If $f(x)+f(z)\geq f(a)$, then $\mathbb{U}(x,z)=f^{(-1)}(\min\{f(x)+f(z),f(a)\})=a=\mathbb{U}(y,z)$.

3.6.1.2 If $x\in I_a$ or $z\in I_a$, then $\mathbb{U}(x,z)=a=\mathbb{U}(y,z)$.

3.6.2 If $(f(y),f(z))\in (f(a),f(e_2)]^2$, then from (iv) and $f(y),f(z)\geq f(a)$, we have

$$\mathbb{U}(x,z)=a\leq f^{(-1)}(\min\{f(y),f(z)\})=\mathbb{U}(y,z).$$

3.6.3 If $(f(y),f(z))\in (f(e_2),+\infty]^2\cup (f(e_2),+\infty]\times (f(a),f(e_2)]\cup (f(a),f(e_2)]\times (f(e_2),+\infty]$, then from (iv) and $f(y),f(z)\geq f(a)$, we have

$$\mathbb{U}(x,z)=a\leq f^{(-1)}(\max\{f(y),f(z)\})=\mathbb{U}(y,z).$$

3.7 $f(a)< f(x)\leq f(y)$.

3.7.1 If $0< f(z)\leq f(a)$, then $\mathbb{U}(x,z)=a=\mathbb{U}(y,z)$.

3.7.2 $f(a)< f(z)\leq f(e_2)$.

3.7.2.1 If $f(a)< f(x)\leq f(y)\leq f(e_2)$, then $\mathbb{U}(x,z)=f^{(-1)}(\min\{f(x),f(z)\})$ and
$$\mathbb{U}(y,z)=f^{(-1)}(\min\{f(y),f(z)\}).$$

3.7.2.1.1 If $f(x)\leq f(y)\leq f(z)$, then $\mathbb{U}(x,z)=x\leq y=\mathbb{U}(y,z)$.

3.7.2.1.2 If $f(z)\leq f(x)\leq f(y)$, then $\mathbb{U}(x,z)=z=\mathbb{U}(y,z)$.

3.7.2.1.3 If $f(x)\leq f(z)\leq f(y)$ then, because of (iv), we have $\mathbb{U}(x,z)=x\leq z=\mathbb{U}(y,z)$.

3.7.2.2 If $f(a)< f(x)\leq f(e_2)< f(y)$, then $\mathbb{U}(x,z)=f^{(-1)}(\min\{f(x),f(z)\})$ and $$\mathbb{U}(y,z)=f^{(-1)}(\max\{f(y),f(z)\})=f^{(-1)}(f(y))=y.$$

3.7.2.2.1 If $f(x)\leq f(z)$, then $\mathbb{U}(x,z)=x\leq y=\mathbb{U}(y,z)$.

3.7.2.2.2 If $f(x)\geq f(z)$ then, because of (iv), we have $\mathbb{U}(x,z)=z\leq y=\mathbb{U}(y,z)$.

3.7.2.3 If $f(e_2)< f(x)\leq f(y),$ then $\mathbb{U}(x,z)=x\leq y=\mathbb{U}(y,z)$.

3.7.3 If $f(z)> f(e_2)$, then $\mathbb{U}(x,z)=f^{(-1)}(\max\{f(x),f(z)\})$ and $\mathbb{U}(y,z)=f^{(-1)}(\max\{f(y),f(z)\})$.

3.7.3.1 If $f(x)\leq f(y)\leq f(z)$, then $\mathbb{U}(x,z)=z=\mathbb{U}(y,z)$.

3.7.3.2 If $f(x)\leq f(z)\leq f(y)$ then, because of (iv), we obtain $\mathbb{U}(x,z)=z\leq y=\mathbb{U}(y,z)$.

3.7.3.3 If $f(z)\leq f(x)\leq f(y)$, then $\mathbb{U}(x,z)=x\leq y=\mathbb{U}(y,z)$.
\item
\textbf{Associativity.} Let $x,y,z\in L.$ We need to verify $%
\mathbb{U}(x,\mathbb{U}(y,z))=\mathbb{U}(\mathbb{U}(x,y),z).$
The proof is split into all possible cases as follows.

1. $f(x)< 0$.

1.1. $f(y)< 0$.

1.1.1. $f(z)< 0$.

1.1.1.1 If $f(x)+f(y)\in Ran(f)$ and $f(y)+f(z)\in Ran(f)$, then
\begin{eqnarray*}
\mathbb{U}(\mathbb{U}(x,y),z)&=&\mathbb{U}(f^{(-1)}(\min\{(f(x)+f(y)), f(a)\}),z)\\
&=&f^{(-1)}(f\circ f^{(-1)}(f(x)+f(y))+f(z))  \mbox{~~~~(since $f(x)+f(y)\leq 0\leq f(a)$)}\\
&=&f^{(-1)}(f(x)+f(y)+f(z)).\\
\mathbb{U}(x,\mathbb{U}(y,z))&=&\mathbb{U}(x,f^{(-1)}(\min\{(f(y)+f(z)),f(a)\})\\
&=&f^{(-1)}(f(x)+f(y)+f(z)).
\end{eqnarray*}
1.1.1.2 If $f(x)+f(y)\in [-\infty, f(0_L))$ and $f(y)+f(z)\in [-\infty, f(0_L))$, then
\begin{eqnarray*}
\mathbb{U}(\mathbb{U}(x,y),z)&=&\mathbb{U}(f^{(-1)}(\min\{(f(x)+f(y)), f(a)\}),z)\\
&=&\mathbb{U}(\wedge L,z) \\
&=&\mathbb{U}(0_L,z).\\
&=&f^{(-1)}(f(z)+f(0_L))\\
&=&0_L.\\
\mathbb{U}(x,\mathbb{U}(y,z))&=&\mathbb{U}(x,f^{(-1)}(\min\{(f(y)+f(z)),f(a)\})\\
&=&\mathbb{U}(x,\wedge L) \\
&=&\mathbb{U}(x,0_L) \\
&=&f^{(-1)}(f(x)+f(0_L))\\
&=&0_L.
\end{eqnarray*}
1.1.1.3 If $f(x)+f(y)\in [-\infty, f(0_L))$ and $f(y)+f(z)\in Ran(f)$, then
\begin{eqnarray*}
\mathbb{U}(\mathbb{U}(x,y),z)&=&\mathbb{U}(f^{(-1)}(\min\{(f(x)+f(y)), f(a)\}),z)\\
&=&\mathbb{U}(\wedge L,z) \\
&=&\mathbb{U}(0_L,z)\\
&=&f^{(-1)}(f(z)+f(0_L))\\
&=&0_L.\\
\mathbb{U}(x,\mathbb{U}(y,z))&=&\mathbb{U}(x,f^{(-1)}(\min\{(f(y)+f(z)),f(a)\})\\
&=&f^{(-1)}(f(x)+f(y)+f(z))\\
&=&\wedge L\\
&=& 0_L.
\end{eqnarray*}
1.1.1.4 If $f(x)+f(y)\in Ran(f)$ and $f(y)+f(z)\in [-\infty, f(0_L))$, then
\begin{eqnarray*}
\mathbb{U}(\mathbb{U}(x,y),z)&=&\mathbb{U}(f^{(-1)}(\min\{(f(x)+f(y)), f(a)\}),z)\\
&=&f^{(-1)}(f(x)+f(y)+f(z)) \\
&=&\wedge L\\
&=& 0_L.\\
\mathbb{U}(x,\mathbb{U}(y,z))&=&\mathbb{U}(x,f^{(-1)}(\min\{(f(y)+f(z)),f(a)\})\\
&=&\mathbb{U}(x,\wedge L)\\
&=&\mathbb{U}(x,0_L)\\
&=&0_L.
\end{eqnarray*}
1.1.2. $0\leq f(z)$.

1.1.2.1 If $f(x)+f(y)\in Ran(f)$, then
\begin{eqnarray*}
\mathbb{U}(\mathbb{U}(x,y),z)&=&\mathbb{U}(f^{(-1)}(\min\{(f(x)+f(y)), f(a)\}),z)\\
&=&f^{(-1)}(\min\{f\circ f^{(-1)}(f(x)+f(y)), f(z)\})  \mbox{~~~~(since $f(x)+f(y)\leq 0\leq f(a)$)}\\
&=&f^{(-1)}(f(x)+f(y)).\\
\mathbb{U}(x,\mathbb{U}(y,z))&=&\mathbb{U}(x,f^{(-1)}(\min\{f(y), f(z)\})\\
&=&f^{(-1)}(\min\{f(x)+f(y), f(a)\})\\
&=&f^{(-1)}(f(x)+f(y)).
\end{eqnarray*}
1.1.2.2 If $f(x)+f(y)\in [-\infty, f(0_L))$, then
\begin{eqnarray*}
\mathbb{U}(\mathbb{U}(x,y),z)&=&\mathbb{U}(f^{(-1)}(\min\{(f(x)+f(y)), f(a)\}),z)\\
&=&\mathbb{U}(\wedge L,z)  \\
&=&0_L.\\
\mathbb{U}(x,\mathbb{U}(y,z))&=&\mathbb{U}(x,f^{(-1)}(\min\{f(y), f(z)\})\\
&=&f^{(-1)}(\min\{f(x)+f(y), f(a)\})\\
&=&\wedge L\\
&=&0_L.
\end{eqnarray*}

1.2. $0\leq f(y)\leq f(a)$.

1.2.1. $f(z)< 0$.

1.2.1.1 If $f(x)+f(z)\in Ran(f)$, then
\begin{eqnarray*}
\mathbb{U}(\mathbb{U}(x,y),z)&=&\mathbb{U}(z,\mathbb{U}(x,y))\\
&=&f^{(-1)}((f(x)+f(z)).    \mbox{~~~~~(by 1.1.2.1)}\\
\mathbb{U}(x,\mathbb{U}(y,z))&=&\mathbb{U}(x,f^{(-1)}(\mbox{min}\{f(y), f(z)\}))\\
&=&f^{(-1)}((f(x)+f(z)).
\end{eqnarray*}
1.2.1.2 If $f(x)+f(z)\in [-\infty, f(0_L))$, then by 1.1.2.2,
\begin{eqnarray*}
\mathbb{U}(\mathbb{U}(x,y),z)&=&\mathbb{U}(z,\mathbb{U}(x,y))=0_L. \\
\mathbb{U}(x,\mathbb{U}(y,z))&=&\mathbb{U}(x,f^{(-1)}(\mbox{min}\{f(y), f(z)\}))\\
&=&f^{(-1)}((f(x)+f(z))\\
&=&\wedge L \\
&=&0_L.
\end{eqnarray*}
1.2.2. If $0\leq f(z)\leq f(a)$, then $\min\{f(y)+f(z),f(a)\}\in Ran(f).$

1.2.2.1 If $y,z\notin I_a$, then
\begin{eqnarray*}
\mathbb{U}(\mathbb{U}(x,y),z)&=&\mathbb{U}(f^{(-1)}(\min\{f(x), f(y)\}),z)\\
&=&f^{(-1)}(\min\{f(x), f(z)\})\\
&=&f^{(-1)}(f(x)).\\
\mathbb{U}(x,\mathbb{U}(y,z))&=&\mathbb{U}(x,f^{(-1)}(\min\{(f(y)+f(z)), f(a)\}))\\
&=&f^{(-1)}(f(x)).
\end{eqnarray*}
1.2.2.2 If $y\in I_a$ or $z\in I_a$, then
$\mathbb{U}(\mathbb{U}(x,y),z)=\mathbb{U}(x,z)=x=\mathbb{U}(x,a)=\mathbb{U}(x,\mathbb{U}(y,z))$.

1.2.3. If $f(a)< f(z)$, then
\begin{eqnarray*}
\mathbb{U}(\mathbb{U}(x,y),z)&=&\mathbb{U}(f^{(-1)}(\min\{f(x), f(y)\},z)\\
&=&f^{(-1)}(\min\{f(x), f(z)\}) \\
&=&f^{(-1)}((f(x)).\\
\mathbb{U}(x,\mathbb{U}(y,z))&=&\mathbb{U}(x,a)\\
&=&f^{(-1)}((f(x)).
\end{eqnarray*}
1.3. $f(a)< f(y)$.

1.3.1. If $f(z)< 0$, then
\begin{eqnarray*}
\mathbb{U}(\mathbb{U}(x,y),z)&=&\mathbb{U}(f^{(-1)}(\mbox{min}\{f(x), f(y)\},z)=\mathbb{U}(x,z). \\
\mathbb{U}(x,\mathbb{U}(y,z))&=&\mathbb{U}(x,f^{(-1)}(\min\{f(y), f(z)\}))=\mathbb{U}(x,z).
\end{eqnarray*}
1.3.2. If $0\leq f(z)\leq f(a)$, then
\begin{eqnarray*}
\mathbb{U}(\mathbb{U}(x,y),z)&=&\mathbb{U}(f^{(-1)}(\min\{f(x), f(y)\})),z)\\
&=&\mathbb{U}(x,z)\\
&=&f^{(-1)}(\min\{f(x), f(z)\})\\
&=&x.\\
\mathbb{U}(x,\mathbb{U}(y,z)&=&\mathbb{U}(x,a)=f^{(-1)}(\min\{f(x), f(a)\})\}=x.
\end{eqnarray*}
1.3.3 $f(z)> f(a).$

1.3.3.1 If $(f(y),f(z))\in (f(a),f(e_2)]^2$, then
\begin{eqnarray*}
\mathbb{U}(\mathbb{U}(x,y),z)&=&\mathbb{U}(f^{(-1)}(\mbox{min}\{f(x), f(y)\})),z)\\
&=&\mathbb{U}(x,z)\\
&=&f^{(-1)}(\mbox{min}\{f(x), f(z)\})\\
&=&x.\\
\mathbb{U}(x,\mathbb{U}(y,z))&=&f^{(-1)}(\min\{f(x),f\circ f^{(-1)}(\min \{f(y),f(z)\})\})=x.
\end{eqnarray*}
1.3.3.2 If $(f(y),f(z))\in [f(e_2),+\infty]^2\cup [f(e_2),+\infty]\times (f(a),f(e_2)]\cup (f(a),f(e_2)]\times [f(e_2),+\infty]$, then
\begin{eqnarray*}
\mathbb{U}(\mathbb{U}(x,y),z)&=&\mathbb{U}(f^{(-1)}(\min\{f(x), f(y)\})),z)\\
&=&\mathbb{U}(x,z)\\
&=&f^{(-1)}(\mbox{min}\{f(x), f(z)\})\\
&=&x.\\
\mathbb{U}(x,\mathbb{U}(y,z)&=&\mathbb{U}(x,f^{(-1)}(\{\max\{f(y),f(z)\})))\\
&=&f^{(-1)}(\min\{f(x), f\circ f^{(-1)}(\{\max\{f(y),f(z)\})\})\}\\
&=&x.
\end{eqnarray*}

2. $0\leq f(x)\leq f(a)$.

2.1. $f(y)< 0$.

2.1.1. $f(z)< 0$.

2.1.1.1 If $f(y)+f(z)\in Ran(f)$ and $f(x)+f(y)\in Ran(f)$, then
\begin{eqnarray*}
\mathbb{U}(\mathbb{U}(x,y),z)&=&\mathbb{U}(z,\mathbb{U}(x,y)) \mbox{~~~~~~~~~~~(by the commutativity of $\mathbb{U}$)}\\
&=&f^{(-1)}(f(z)+f(y)). \mbox{~~~(by 1.2.1.1)}\\
\mathbb{U}(x,\mathbb{U}(y,z))&=&\mathbb{U}(\mathbb{U}(y,z),x)\mbox{~~~~~~~~~~~~(by the commutativity of $\mathbb{U}$)}\\
&=&f^{(-1)}(f(z)+f(y)). \mbox{~~~(by 1.1.2.1)}
\end{eqnarray*}
2.1.1.2 If $f(y)+f(z)\in [-\infty, f(0_L))$, then
\begin{eqnarray*}
\mathbb{U}(\mathbb{U}(x,y),z)&=&\mathbb{U}(z,\mathbb{U}(x,y)) \mbox{~~~~~~~~~~(by the commutativity of $\mathbb{U}$)}\\
&=&0_L. \mbox{~~~~~~~~~~~~~~~~~~~~~~~(by 1.2.1.2)}\\
\mathbb{U}(x,\mathbb{U}(y,z))&=&\mathbb{U}(\mathbb{U}(y,z),x)\mbox{~~~~~~~~~~~(by the commutativity of $\mathbb{U}$)}\\
&=&f^{(-1)}(f(z)+f(y)) \mbox{~~~(by 1.1.2.1)}\\
&=&\wedge L\\
&=& 0_L.
\end{eqnarray*}
2.1.2. If $0\leq f(z)\leq f(a)$, then
\begin{eqnarray*}
\mathbb{U}(\mathbb{U}(x,y),z)&=&\mathbb{U}(f^{(-1)}(\mbox{min}\{f(x),f(y)\}),z),\\
&=&\mathbb{U}(f^{(-1)}(f(y)),z)\\
&=&\mathbb{U}(y,z)\\
&=&f^{(-1)}(\mbox{min}\{f(y),f(z)\})\\
&=&y.\\
\mathbb{U}(x,\mathbb{U}(y,z))&=&\mathbb{U}(x,f^{(-1)}(\mbox{min}\{f(y),f(z)\}))\\
&=&\mathbb{U}(x,y)\\
&=&f^{(-1)}(\mbox{min}\{f(x),f(y)\})\\
&=&y.
\end{eqnarray*}
2.1.3. If $f(z)> f(a)$, then
\begin{eqnarray*}
\mathbb{U}(\mathbb{U}(x,y),z)&=&\mathbb{U}(f^{(-1)}(\mbox{min}\{f(x),f(y)\}),z),\\
&=&\mathbb{U}(f^{(-1)}(f(y)),z)\\
&=&\mathbb{U}(y,z)\\
&=&f^{(-1)}(\mbox{min}\{f(y),f(z)\})\\
&=&y.\\
\mathbb{U}(x,\mathbb{U}(y,z))&=&\mathbb{U}(x,f^{(-1)}(\mbox{min}\{f(y),f(z)\}))\\
&=&\mathbb{U}(x,y)\\
&=&f^{(-1)}(\mbox{min}\{f(x),f(y)\})\\
&=&y.
\end{eqnarray*}

2.2. $0\leq f(y)\leq f(a)$.

2.2.1. If $f(z)< 0$, then
\begin{eqnarray*}
\mathbb{U}(\mathbb{U}(x,y),z)&=&\mathbb{U}(z,\mathbb{U}(x,y)) \mbox{~~~~~~~~~~(by the commutativity of $\mathbb{U}$)}\\
&=&f^{(-1)}(f(z)). \mbox{~~~~~~~~~~~(by 1.2.2.1-1.2.2.2)}\\
\mathbb{U}(x,\mathbb{U}(y,z))&=&\mathbb{U}(x,f^{(-1)}(\mbox{min}\{f(y),f(z)\}))\\
&=&\mathbb{U}(x,z)\\
&=&f^{(-1)}(\mbox{min}\{f(x),f(z)\})\\
&=&f^{(-1)}(f(z)).
\end{eqnarray*}
2.2.2. If $0\leq f(z)\leq f(a)$, then both $\min\{f(x)+f(y),f(a)\}\in Ran(f)$ and $\min\{f(y)+f(z),f(a)\}$ belong to $Ran(f)$.

2.2.2.1 If $x,y,z\notin I_a$, then
\begin{eqnarray*}
\mathbb{U}(\mathbb{U}(x,y),z)&=&\mathbb{U}(f^{(-1)}(\mbox{min}\{f(x)+f(y),f(a)\}),z)\\
&=&f^{(-1)}(\mbox{min}\{f\circ f^{(-1)}(\mbox{min}\{f(x)+f(y),f(a)\})+f(z),f(a)\})\\
&=&f^{(-1)}(\mbox{min}\{(\mbox{min}\{f(x)+f(y),f(a)\})+f(z),f(a)\})\\
&=&f^{(-1)}(\mbox{min}\{\mbox{min}\{(f(x)+f(y)+f(z),f(a)+f(z)\},f(a)\})\\
&=&f^{(-1)}(\mbox{min}\{f(x)+f(y)+f(z),f(a)+f(z),f(a)\})\\
&=&f^{(-1)}(\mbox{min}\{f(x)+f(y)+f(z),f(a)\}).   \mbox{~~~~~(since $f(a)+f(z)\geq f(a)$)}.\\
\mathbb{U}(x,\mathbb{U}(y,z))&=&\mathbb{U}(\mathbb{U}(y,z),x) \mbox{~~~~~~~~~~~~~~~~~~~~~~~~~~~~~~~~~~~~~~(by the commutativity of $\mathbb{U}$)}\\
&=&f^{(-1)}(\mbox{min}\{f(x)+f(y)+f(z),f(a)\}).  \mbox{~~~~~(by 2.2.2.)}
\end{eqnarray*}
2.2.2.2 If $x,y,z\in I_a$ or $x,z\in I_a$, then $\mathbb{U}(\mathbb{U}(x,y),z)=a=\mathbb{U}(x,\mathbb{U}(y,z))$.

2.2.2.3 If $y,z\in I_a$ and $x\notin I_a$, then
$$\mathbb{U}(\mathbb{U}(x,y),z)=a=f^{(-1)}(\min\{f(a)+f(x),f(a)\})=\mathbb{U}(x,\mathbb{U}(y,z)).$$

2.2.2.4 The case $x,y\in I_a$ and $z\notin I_a$ is analogous to 2.2.2.3.\\
2.2.2.5 If $x,y\notin I_a$ and $z\in I_a$, then
\begin{eqnarray*}
\mathbb{U}(\mathbb{U}(x,y),z)=\mathbb{U}(f^{(-1)}(\min\{f(x)+f(y),f(a)\},z)=a=\mathbb{U}(x,a)=\mathbb{U}(x,\mathbb{U}(y,z)).
\end{eqnarray*}
2.2.2.6 Both the case $x\in I_a$ and $y,z\notin I_a$ and the case $y\in I_a$ and $x,z\notin I_a$ are analogous to 2.2.2.5.\\
2.2.3. $f(a)< f(z)$.

2.2.3.1 If $x,y\notin I_a$, then from $f^{(-1)}(\min\{f(x)+f(y),f(a)\})\in [0,f(a)]$, we have
\begin{eqnarray*}
\mathbb{U}(\mathbb{U}(x,y),z)=\mathbb{U}(f^{(-1)}(\mbox{min}\{f(x)+f(y),f(a)\}),z)=a=\mathbb{U}(x,a)=\mathbb{U}(x,\mathbb{U}(y,z)).
\end{eqnarray*}

2.2.3.2 If $x,y\in I_a$, then $\mathbb{U}(\mathbb{U}(x,y),z)=\mathbb{U}(a,z)=a=\mathbb{U}(x,a)=\mathbb{U}(x,\mathbb{U}(y,z))$.

2.2.3.3 If $x\notin I_a$ and $y\in I_a$, then $\mathbb{U}(\mathbb{U}(x,y),z)=\mathbb{U}(a,z)=a=\mathbb{U}(x,a)=\mathbb{U}(x,\mathbb{U}(y,z))$.

2.2.3.4 If $x\in I_a$ and $y\notin I_a$, then
\begin{eqnarray*}
\mathbb{U}(\mathbb{U}(x,y),z)=\mathbb{U}(a,z)=a=\mathbb{U}(x,a)=\mathbb{U}(x,\mathbb{U}(y,z)).
\end{eqnarray*}

2.3 $f(a)< f(y)$.

2.3.1 If $f(z)<0$, then
\begin{eqnarray*}
\mathbb{U}(\mathbb{U}(x,y),z)=\mathbb{U}(a,z)=z=\mathbb{U}(x,f^{(-1)}(\mbox{min}\{f(y),f(z))\})=\mathbb{U}(x,\mathbb{U}(y,z)).
\end{eqnarray*}
2.3.2. If $0\leq f(z)\leq f(a)$, then from 2.2.3.1-2.2.3.4, we have
\begin{eqnarray*}
\mathbb{U}(\mathbb{U}(x,y),z)=\mathbb{U}(z,\mathbb{U}(x,y))=a=\mathbb{U}(x,a)=\mathbb{U}(x,\mathbb{U}(y,z)).
\end{eqnarray*}
2.3.3. $f(a)< f(z)$.

2.3.3.1 If $(f(y),f(z))\in (f(a),f(e_2)]^2$, then $$\mathbb{U}(\mathbb{U}(x,y),z)=\mathbb{U}(a,z)=a=\mathbb{U}(x,f^{(-1)}(\min\{f(y),f(z)\}))=\mathbb{U}(x,\mathbb{U}(y,z)).$$

2.3.3.2 If $(f(y),f(z))\in [f(e_2),+\infty]^2\cup [f(e_2),+\infty]\times (f(a),f(e_2)]\cup (f(a),f(e_2)]\times [f(e_2),+\infty]$, then
\begin{eqnarray*}
\mathbb{U}(\mathbb{U}(x,y),z)=\mathbb{U}(a,z)=a=\mathbb{U}(x,f^{(-1)}(\max\{f(y),f(z)\}))=\mathbb{U}(x,\mathbb{U}(y,z)).
\end{eqnarray*}
3. $f(a)< f(x)$.

3.1. $f(y)< 0$.

3.1.1. If $f(z)< 0$, then
\begin{eqnarray*}
\mathbb{U}(\mathbb{U}(x,y),z)&=&\mathbb{U}(z,\mathbb{U}(x,y))   \mbox{~~~~~~~~~~(by the commutativity of $\mathbb{U}$)}\\
&=&\mathbb{U}(z,y)      \mbox{~~~~~~~~~~~~~~~~~(by 1.3.1.)}\\
&=&f^{(-1)}(f(y)+f(z)). \\
\mathbb{U}(x,\mathbb{U}(y,z))&=&\mathbb{U}(\mathbb{U}(y,z),x)   \mbox{~~~~~~~~~~(by the commutativity of $\mathbb{U}$)}\\
&=&f^{(-1)}(f(y)+f(z)).    \mbox{~~~~~~(by 1.1.2.)}
\end{eqnarray*}
3.1.2. If $0\leq f(z)\leq f(a)$, then
\begin{eqnarray*}
\mathbb{U}(\mathbb{U}(x,y),z)&=&\mathbb{U}(z,\mathbb{U}(x,y))   \mbox{~~~~~~(by the commutativity of $\mathbb{U}$)}\\
&=&\mathbb{U}(z,y)      \mbox{~~~~~~~~~~~~(by 2.3.1.)}\\
&=&f^{(-1)}(\min\{f(z),f(y)\})\\
&=&y.\\
\mathbb{U}(x,\mathbb{U}(y,z))&=&\mathbb{U}(\mathbb{U}(y,z),x)   \mbox{~~~~~(by the commutativity of $\mathbb{U}$)}\\
&=&y.   \mbox{~~~~(by 1.2.3.)}
\end{eqnarray*}
3.1.3 If $f(a)< f(z)$, then
\begin{eqnarray*}
\mathbb{U}(\mathbb{U}(x,y),z)&=&\mathbb{U}(y,z)\\
&=&f^{(-1)}(\mbox{min}\{f(y),f(z)\})\\
&=&y.\\
\mathbb{U}(x,\mathbb{U}(y,z))&=&\mathbb{U}(\mathbb{U}(y,z),x)   \mbox{~~~~~(by the commutativity of $\mathbb{U}$)}\\
&=&y.   \mbox{~~~~(by 1.3.3.)}
\end{eqnarray*}

3.2. $0\leq f(y)\leq f(a)$.

3.2.1. If $f(z)< 0$, then $\mathbb{U}(\mathbb{U}(x,y),z)=\mathbb{U}(a,z)=z=\mathbb{U}(x,z)=\mathbb{U}(x,\mathbb{U}(y,z))$.

3.2.2. If $0\leq f(z)\leq f(a)$, then by 2.3.2, 2.2.3 and the commutativity of $\mathbb{U}$, we have
\begin{eqnarray*}
\mathbb{U}(\mathbb{U}(x,y),z)=\mathbb{U}(z,\mathbb{U}(x,y))=a=\mathbb{U}(\mathbb{U}(y,z),x)=\mathbb{U}(x,\mathbb{U}(y,z)).
\end{eqnarray*}
3.2.3. If $f(a)< f(z)$, then $\mathbb{U}(\mathbb{U}(x,y),z)=\mathbb{U}(a,z)=a=\mathbb{U}(x,a)=\mathbb{U}(x,\mathbb{U}(y,z))$.

3.3. $f(a)< f(y)$.

3.3.1. If $f(z)< 0$, then
\begin{eqnarray*}
\mathbb{U}(\mathbb{U}(x,y),z)&=&\mathbb{U}(z,\mathbb{U}(x,y)   \mbox{~~~~~(by the commutativity of $\mathbb{U}$)}\\
&=&z.      \mbox{~~~~~~~~~~~~~~~~~~(by 1.3.3.)}\\
\mathbb{U}(x,\mathbb{U}(y,z))&=&\mathbb{U}(\mathbb{U}(y,z),x)   \mbox{~~~~(by the commutativity of $\mathbb{U}$)}\\
&=&z.   \mbox{~~~~~~~~~~~~~~~~~~~(by 3.1.3.)}
\end{eqnarray*}
3.3.2. If $0\leq f(z)\leq f(a)$, then
\begin{eqnarray*}
\mathbb{U}(\mathbb{U}(x,y),z)&=&\mathbb{U}(z,\mathbb{U}(x,y)   \mbox{~~~~~~(by the commutativity of $\mathbb{U}$)}\\
&=&a.      \mbox{~~~~~~~~~~~~~~~~~~~(by 2.3.3.)}\\
\mathbb{U}(x,\mathbb{U}(y,z))&=&\mathbb{U}(\mathbb{U}(y,z),x)   \mbox{~~~~~(by the commutativity of $\mathbb{U}$)}\\
&=&a.      \mbox{~~~~~~~~~~~~~~~~~~~~(by 3.2.3.)}
\end{eqnarray*}
3.3.3. $f(a)< f(z)$.

3.3.3.1 If $f(x),f(y),f(z)\notin (f(e_2),+\infty]$, then
\begin{eqnarray*}
\mathbb{U}(\mathbb{U}(x,y),z)&=&\mathbb{U}(f^{(-1)}(\mbox{min}\{f(x),f(y)\}),z)\\
&=&f^{(-1)}(\mbox{min}\{f\circ f^{(-1)}(\mbox{min}\{f(x),f(y)\}),f(z)\})\\
&=&f^{(-1)}(\mbox{min}\{\mbox{min}\{f(x),f(y)\},f(z)\})\\
&=&f^{(-1)}(\mbox{min}\{f(x),f(y),f(z)\}).\\
\mathbb{U}(x,\mathbb{U}(y,z))&=&\mathbb{U}(\mathbb{U}(y,z),x)  \mbox{~~~~~~~~~~~~~~~~~~~~~~~~~(by the commutativity of $\mathbb{U}$)}\\
&=&f^{(-1)}(\mbox{min}\{f(x),f(y),f(z)\}).   \mbox{~~~~(by 3.3.3.)}\\
\end{eqnarray*}
3.3.3.2 If exactly one of $f(x),f(y),f(z)$ belongs to $(f(e_2),+\infty]$, say $f(x)\in (f(e_2),+\infty]$, then
\begin{eqnarray*}
\mathbb{U}(\mathbb{U}(x,y),z)&=&\mathbb{U}(f^{(-1)}(\max\{f(x),f(y)\}),z)\\
&=&U(x,z)\\
&=&f^{(-1)}(\max\{f(x),f(z)\})\\
&=&x.\\
\mathbb{U}(x,\mathbb{U}(y,z))&=&\mathbb{U}(x,f^{(-1)}(\min\{f(y),f(z)\}))  \\
&=&f^{(-1)}(\max\{f(x),f\circ f^{(-1)}(\min\{f(y),f(z)\})\})   \\
&=&x.
\end{eqnarray*}

3.3.3.3 If at least two of $f(x),f(y),f(z)$ belong to $(f(e_2),+\infty]$, then
\begin{eqnarray*}
\mathbb{U}(\mathbb{U}(x,y),z)&=&\mathbb{U}(f^{(-1)}(\max\{f(x),f(y)\}),z)\\
&=&f^{(-1)}(\max\{f\circ f^{(-1)}(\max\{f(x),f(y)\}),f(z)\})\\
&=&f^{(-1)}(\max\{f(x),f(y),f(z)\}.\\
\mathbb{U}(x,\mathbb{U}(y,z))&=&\mathbb{U}(x,f^{(-1)}(\max\{f(y),f(z)\}))  \\
&=&f^{(-1)}(\max\{f(x),f\circ f^{(-1)}(\max\{f(y),f(z)\})\})   \\
&=&f^{(-1)}(\max\{f(x),f(y),f(z)\}.
\end{eqnarray*}
Hence, $\mathbb{U}$ is associative.
\end{enumerate}
Consequently, $\mathbb{U}$ is a 2-uninorm on $L.$
\end{proof}

The following theorem is a dual consequence of Theorem \ref{thm-1}.
\begin{theorem}\label{thm-2}
Let $e_1, a, e_2\in L$ with $0_L\leq e_1\leq a\leq e_2\leq 1_L$, and $f:L\to [-\infty,+\infty]$ be an injective decreasing function with $f(e_1)=0$.
If $f$ satisfies the following five conditions: for all $x,y\in L$,
\begin{enumerate}[{\rm (i)}]
  \item
if $(f(x),f(y))\in [f(a),0]^2\cup [0,+\infty]^2$ then
$$\max\{f(x)+f(y),f(a)\}\in Ran(f)\cup (f(0_L),+\infty],$$
 \item
for all $(f(x),f(y))\in [f(a),0]^2\cup [0,+\infty]^2$, if $f(x)$ and $f(y)$ have at least one same summand $f(z)\in Ran(f)$, then $x\nparallel y$,
\item
if $f(a)\leq f(x)<0<f(y)$ and $e_1<y$, then $x\nparallel y$,
  \item
if $f(a)\geq f(x)> f(y)$ or $f(x)\leq f(a)<0<f(y)$, then $x\nparallel y$,
  \item
for all $f(x),f(y)\in (f(a),0)$, if $f(x)+f(y)\in Ran(f)$ and $f(a)\leq f(x)+f(y)< 0$ then $f^{-1}(f(x)+f(y))\nparallel a$,
\end{enumerate}
then the following function $\mathbb{U}_d: L^2\to L$ is a 2-uninorm, and we call $f$ an additive generator of $\mathbb{U}_d$. For all $x,y\in L$,
\begin{equation}\label{eq-2}
\mathbb{U}_d(x, y)=
\begin{cases}
f^{(-1)}(f(x)+f(y)), & (f(x),f(y))\in [0,+\infty]^2, \\
f^{(-1)}(\max\{f(x)+f(y), f(a)\}), & (f(x),f(y))\in [f(a),0]^2 \mbox{ and } x,y\notin I_a, \\
a, & (f(x),f(y))\in [-\infty,f(a))\times [f(a),0]\\
&~~~~~\cup [f(a),0]\times [-\infty,f(a))\\
&~~~~~\cup [f(a),0]\times \mathcal{G}_a\cup \mathcal{G}_a\times [f(a),0],\\
f^{(-1)}(\min\{f(x),f(y)\}), &  (f(x),f(y))\in [-\infty,f(e_2))^2\cup [-\infty,f(e_2))\\
&~~~~~\times [f(e_2),f(a))\cup [f(e_2),f(a))\times [-\infty,f(e_2)),\\
f^{(-1)}(\max\{f(x),f(y)\}), & otherwise.
\end{cases}
\end{equation}
\end{theorem}

 From Remark \ref{remar2.2}, we have the following corollary.
\begin{corollary}
\begin{enumerate}[{\rm (i)}]
  \item
Taking $e_2=1_L$ in Theorem \ref{thm-1}, we obtain the uni-nullnorm $U_N$ as follows. For all $x,y\in L$,
\begin{equation*}
U_N(x,y)=
\begin{cases}
f^{(-1)}(f(x)+f(y)), & (f(x),f(y))\in  [-\infty, 0]^2, \\
f^{(-1)}(\min\{f(x)+f(y), f(a)\}), & (f(x),f(y))\in [0,f(a)]^2 \mbox{ and } x,y\notin I_a, \\
a, & (f(x),f(y))\in [f(a),+\infty)\times [0,f(a)]\\
&~~~~~~~~~~~~~~\cup [0,f(a)]\times [f(a),+\infty)\\
&~~~~~~~~~~~~~~\cup [0,f(a)]\times \mathcal{F}_a\cup \mathcal{F}_a\times [0,f(a)],\\
f^{(-1)}(\min\{f(x),f(y)\}), & \text{ otherwise}.
\end{cases}
\end{equation*}
  \item
Taking $e_2=1_L$ in Theorem \ref{thm-2}, we have the uni-nullnorm $U_N$ as follows. For all $x,y\in L$,
\begin{equation*}
U_N(x,y)=
\begin{cases}
f^{(-1)}(f(x)+f(y)), & (f(x),f(y))\in (0,+\infty]^2,\\
f^{(-1)}(\max\{f(x)+f(y), f(a)\}), & (f(x),f(y))\in [f(a),0]^2 \mbox{ and } x,y\notin I_a, \\
a, & (f(x),f(y))\in [-\infty,f(a))\times [f(a),0]\\
&~~~~~~~~~~~~~~\cup [f(a),0]\times [-\infty,f(a))\\
&~~~~~~~~~~~~~~\cup [f(a),0]\times \mathcal{G}_a\cup \mathcal{G}_a\times [f(a),0],\\
f^{(-1)}(\max\{f(x),f(y)\}), & \text{ otherwise}.
\end{cases}
\end{equation*}
\item
Taking $e_1=0_L$ in Theorem \ref{thm-1}, we get the null-uninorm as below. For all $x,y\in L$,
\begin{equation*}
N_U(x,y)=
\begin{cases}
f^{(-1)}(\min\{f(x)+f(y), f(a)\}), & (f(x),f(y))\in [0,f(a)]^2\mbox{ and } x,y\notin I_a, \\
a, & (f(x),f(y))\in [f(a),+\infty)\times [0,f(a)]\\
&~~~~~~~~~~~~~~\cup [0,f(a)]\times [f(a),+\infty)\\
&~~~~~~~~~~~~~~\cup [0,f(a)]\times \mathcal{F}_a\cup \mathcal{F}_a\times [0,f(a)],\\
f^{(-1)}(\min\{f(x),f(y)\}), & (f(x),f(y))\in [f(a),f(e_2)]^2, \\
f^{(-1)}(\max\{f(x),f(y)\}), & \text{ otherwise}.
\end{cases}
\end{equation*}
In this case, $f$ is an injective increasing function from $L$ to $[0,+\infty]$ with $f(0_L)=0$ and it satisfies the conditions (i), (ii), (v) in Theorem \ref{thm-1} and\\
\emph{(iv')} if $f(a)\leq f(x)<f(y)$, then $x\nparallel y$.
  \item
Taking $e_1=0_L$ in Theorem \ref{thm-2}, we have the null-uninorm as below. For all $x,y\in L$,
\begin{equation*}
N_U(x,y)=
\begin{cases}
f^{(-1)}(\max\{f(x)+f(y), f(a)\}), & (f(x),f(y))\in [f(a),0]^2 \mbox{ and } x,y\notin I_a, \\
a, & (f(x),f(y))\in [-\infty,f(a))\times [f(a),0]\\
&~~~~~~~~~~~~~~\cup [f(a),0]\times [-\infty,f(a))\\
&~~~~~~~~~~~~~~\cup [f(a),0]\times \mathcal{G}_a\cup \mathcal{G}_a\times [f(a),0],\\
f^{(-1)}(\max\{f(x),f(y)\}), & (f(x),f(y))\in [f(e_2),f(a)]^2, \\
f^{(-1)}(\min\{f(x),f(y)\}), & \text{ otherwise}.
\end{cases}
\end{equation*}
In this case, $f$ is an injective decreasing function from $L$ to $[-\infty,0]$ with $f(0_L)=0$ and it satisfies the conditions (i), (ii), (v) in Theorem \ref{thm-2} and\\
\emph{(iv{''})} if $f(x)<f(y)\leq f(a)$, then $x\nparallel y$.
  \item
Taking $e_1=0_L$ and $e_2=1_L$ in Theorem \ref{thm-1}, we get the nullnorm as follows. For all $x,y\in L$,
\begin{equation*}
N(x,y)=
\begin{cases}
f^{(-1)}(\min\{f(x)+f(y), f(a)\}), & (f(x),f(y))\in [0,f(a)]^2 \mbox{ and } x,y\notin I_a,\\
a, & (f(x),f(y))\in [f(a),+\infty)\times [0,f(a)]\\
&~~~~~~~~~~~~~~\cup [0,f(a)]\times [f(a),+\infty)\\
&~~~~~~~~~~~~~~\cup [0,f(a)]\times \mathcal{F}_a\cup \mathcal{F}_a\times [0,f(a)],\\
f^{(-1)}(\min\{f(x),f(y)\}), & \text{ otherwise}.
\end{cases}
\end{equation*}
In this case, $f$ is an injective increasing function from $L$ to $[0,+\infty]$ with $f(0_L)=0$ and it satisfies the conditions (i), (ii), (v) in Theorem \ref{thm-1} and (iv').
 \item
Taking $e_1=0_L$ and $e_2=1_L$ in Theorem \ref{thm-2}, we obtain the nullnorm as below. For all $x,y\in L$,
\begin{equation*}
N(x,y)=
\begin{cases}
f^{(-1)}(\max\{f(x)+f(y), f(a)\}), & (f(x),f(y))\in [f(a),0]^2 \mbox{ and } x,y\notin I_a,\\
a, & (f(x),f(y))\in [-\infty,f(a))\times [f(a),0]\\
&~~~~~~~~~~~~~~\cup [f(a),0]\times [-\infty,f(a))\\
&~~~~~~~~~~~~~~\cup [f(a),0]\times \mathcal{G}_a\cup \mathcal{G}_a\times [f(a),0],\\
f^{(-1)}(\max\{f(x),f(y)\}), & \text{ otherwise}.
\end{cases}
\end{equation*}
In this case, $f$ is an injective decreasing function from $L$ to $[-\infty,0]$ with $f(0_L)=0$ and it satisfies the conditions (i), (ii), (v) in Theorem \ref{thm-2} and (iv{''}).
  \item
Taking $a=1_L$ in Theorem \ref{thm-1}, we have the uninorm $U$ in \cite{HW2021} as follows. For all $x,y\in L$,
\begin{equation}\label{eq-u}
U(x,y)=
\begin{cases}
f^{(-1)}(f(x)+f(y)), & (f(x),f(y))\in [-\infty,0]^2\cup [0,+\infty]^2,\\
f^{(-1)}(\min\{f(x),f(y)\}), & \text{ otherwise}.
\end{cases}
\end{equation}
In this case, $f$ just needs to satisfy the conditions (i), (ii) and (iii). Subsequently, we further obtain the t-norm and t-conorm given in \cite{SL2021} by taking $e_1=1_L$ and $e_1=0_L$ in Formula \eqref{eq-u}, respectively.
  \item
Taking $a=1_L$ in Theorem \ref{thm-2}, we deduce the uninorm $U$ as below. For all $x,y\in L$,
\begin{equation*}
U(x,y)=
\begin{cases}
f^{(-1)}(f(x)+f(y)), & (f(x),f(y))\in [-\infty,0]^2\cup [0,+\infty]^2,\\
f^{(-1)}(\max\{f(x),f(y)\}), & \text{ otherwise}.
\end{cases}
\end{equation*}
In this case, $f$ just needs to satisfy the conditions (i), (ii) and (iii).
 \end{enumerate}
\end{corollary}

\begin{remark}
\emph{
\begin{enumerate}[{\rm (i)}]
  \item
Since we require the functions $f$ in both Theorems \ref{thm-1} and \ref{thm-2} to be injective, it is obviously impossible to choose a suitable
 $f:L\to [-\infty,+\infty]$ when the cardinality of $L$ is strictly greater than $\aleph_1.$
  \item
The following are two alternative forms of \eqref{eq-1} and \eqref{eq-2}, respectively. For all $x,y\in L$,
$$
\mathbb{U}(x,y)=
\begin{cases}
f^{-1}(f(x)+f(y)), &\mbox{either } (f(x),f(y))\in [-\infty,0]^2 \mbox{ and } f(x)+f(y)\in Ran(f)\\
& \mbox{or }(f(x),f(y))\in [0,f(a)]^2,  f(x)+f(y)\leq f(a), f(x)+f(y)\in Ran(f) \\
& \mbox{ and } x,y\notin I_a;\\
0_L, &(f(x),f(y))\in [-\infty,0]^2 \mbox{ and }  f(x)+f(y)\in [-\infty,f(0_L));\\
a, &\mbox{either }(f(x),f(y))\in [0,f(a)]^2 \mbox{ and }  f(x)+f(y)\geq f(a),\\
& \mbox{or }(f(x),f(y))\in [f(a),+\infty]\times [0,f(a)] \cup [0,f(a)]\times [f(a),+\infty]\\
& \mbox{or }(f(x),f(y))\in [0,f(a)]\times \mathcal{F}_a\cup \mathcal{F}_a\times [0,f(a)];\\
x, &\mbox{either }(f(x),f(y))\in [-\infty,0]\times [0,+\infty],\\
& \mbox{or }(f(x),f(y))\in [f(e_2),+\infty]\times [f(a),f(e_2)],\\
& \mbox{or }(f(x),f(y))\in [f(a),f(e_2)]^2 \mbox{ and }  f(x)\leq f(y)\\
& \mbox{or }(f(x),f(y))\in [f(e_2),+\infty]^2 \mbox{ and }  f(x)\geq f(y);\\
y, & \mbox{either }(f(x),f(y))\in [0,+\infty]\times [-\infty,0],\\
 & \mbox{or }(f(x),f(y))\in [f(a),f(e_2)]\times [f(e_2),+\infty],\\
& \mbox{or }(f(x),f(y))\in [f(a),f(e_2)]^2 \mbox{ and }  f(x)\geq f(y)\\
& \mbox{or }(f(x),f(y))\in [f(e_2),+\infty]^2 \mbox{ and }  f(x)\leq f(y).
\end{cases}
$$
For all $x,y\in L$,
$$
\mathbb{U}_d(x,y)=
\begin{cases}
f^{-1}(f(x)+f(y)), &\mbox{either }(f(x),f(y))\in [0,+\infty]^2 \mbox{ and } f(x)+f(y)\in Ran(f)\\
& \mbox{or }(f(x),f(y))\in [f(a),0]^2 \mbox{ and }  f(x)+f(y)\geq f(a) \mbox{ and }  \\
&f(x)+f(y)\in Ran(f) \mbox{ and } x,y\notin I_a;\\
0_L, & (f(x),f(y))\in [0,+\infty]^2 \mbox{ and }  f(x)+f(y)\in [f(0_L),+\infty);\\
a, & \mbox{either }(f(x),f(y))\in [f(a),f(e_1)]\times [-\infty ,f(a)]\\
&~~~~~~~~~~~~~~~~~~~~~~~~~~~\cup [-\infty, f(a)]\times [f(a),f(e_1)],\\
& \mbox{or }(f(x),f(y))\in [0,f(a)]^2 \mbox{ and } x\notin I_a \mbox{ and } y\notin I_a\\
& \mbox{or }(f(x),f(y))\in [f(a),0]\times \mathcal{F}_a\cup \mathcal{F}_a\times [f(a),0];\\
x, &\mbox{either }(f(x),f(y))\in [0,+\infty]\times [-\infty,0],\\
& \mbox{or }(f(x),f(y))\in [f(e_2),f(a)]\times [-\infty,f(e_2)],\\
& \mbox{or }(f(x),f(y))\in [f(e_2),f(a)]^2 \mbox{ and } f(x)\geq f(y)\\
& \mbox{or }(f(x),f(y))\in [-\infty,f(e_2)]^2 \mbox{ and } f(x)\geq f(y);\\
y, &\mbox{either }(f(x),f(y))\in [-\infty,0]\times [0,+\infty],\\
& \mbox{or }(f(x),f(y))\in [-\infty,f(e_2)]\times [f(e_2),f(a)],\\
& \mbox{or }(f(x),f(y))\in [f(e_2),f(a)]^2 \mbox{ and } f(x)\leq f(y)\\
& \mbox{or }(f(x),f(y))\in [-\infty,f(e_2)]^2 \mbox{ and } f(x)\leq f(y).
\end{cases}
$$
\item
If $L=[0,1]$, then $f$ in Theorem \ref{thm-1} only needs to satisfy the condition (i) since $x\nparallel y$ for any $x,y\in [0,1]$, and the following function $\mathbb{U}(x,y)$ is a 2-uninorm on $[0,1]$. For all $x,y\in [0,1]$,
\begin{equation}
\mathbb{U}(x, y)=
\begin{cases}
f^{(-1)}(\min\{f(x)+f(y), f(a)\}), & (f(x),f(y))\in [-\infty, 0]^2\cup [0,f(a)]^2; \\
a, & (f(x),f(y))\in (f(a),+\infty]\times [0,f(a)]\\
&~~~~\cup [0,f(a)]\times (f(a),+\infty];\\
f^{(-1)}(\max\{f(x),f(y)\}, & (f(x),f(y))\in (f(e_2),+\infty]^2\cup (f(e_2),+\infty]\\
&~~~~\times (f(a),f(e_2)]\cup (f(a),f(e_2)]\times (f(e_2),+\infty];\\
f^{(-1)}(\min\{f(x),f(y)\}), & \mbox{otherwise}.
\end{cases}
\end{equation}
\end{enumerate}
}
\end{remark}

The following two examples illustrate Theorem \ref{thm-1}.
\begin{example}
\emph{
Consider the lattice $L_1=\{0_{L_1},x_1,x_2,x_3,x_4,e_1,x_5,x_6,a,x_7,x_8,x_9,e_2,x_{10},x_{11},1_{L_1}\}$ given in Fig. \ref{FIG:eg1} and the injective increasing function $f$ defined by Table \ref{TB:1}. One can check that
the function $f$ satisfies Theorem \ref{thm-1}, and the 2-uninorm $\mathbb{U}$ is shown by Table \ref{TB:3}.}
\end{example}
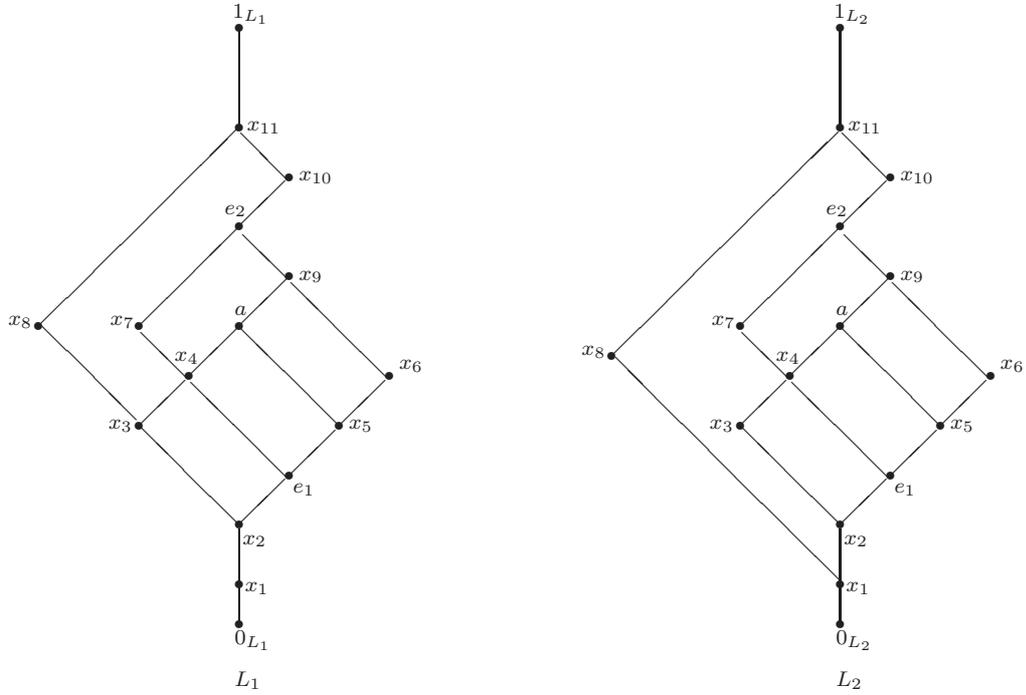
\begin{figure}[!h]
\par\noindent\vskip50pt
 \begin{minipage}{11pc}
\setlength{\unitlength}{0.75pt}\begin{picture}(600,290)
\put(200,60){\circle*{4}}\put(203,58){\makebox(0,0)[l]{\footnotesize $x_1$}}
\put(200,90){\circle*{4}}\put(202,82){\makebox(0,0)[l]{\footnotesize $x_2$}}
\put(225,115){\circle*{4}}\put(227,108){\makebox(0,0)[l]{\footnotesize $e_1$}}
\put(150,140){\circle*{4}}\put(135,140){\makebox(0,0)[l]{\footnotesize $x_3$}}
\put(150,190){\circle*{4}}\put(136,192){\makebox(0,0)[l]{\footnotesize $x_7$}}
\put(175,165){\circle*{4}}\put(168,174){\makebox(0,0)[l]{\footnotesize $x_4$}}
\put(250,140){\circle*{4}}\put(255,140){\makebox(0,0)[l]{\footnotesize $x_5$}}
\put(200,190){\circle*{4}}\put(198,198){\makebox(0,0)[l]{\footnotesize $a$}}
\put(200,240){\circle*{4}}\put(193,248){\makebox(0,0)[l]{\footnotesize $e_2$}}
\put(100,190){\circle*{4}}\put(85,192){\makebox(0,0)[l]{\footnotesize $x_8$}}
\put(275,165){\circle*{4}}\put(280,170){\makebox(0,0)[l]{\footnotesize $x_6$}}
\put(225,215){\circle*{4}}\put(230,215){\makebox(0,0)[l]{\footnotesize $x_9$}}
\put(225,265){\circle*{4}}\put(230,265){\makebox(0,0)[l]{\footnotesize $x_{10}$}}
\put(200,290){\circle*{4}}\put(204,290){\makebox(0,0)[l]{\footnotesize $x_{11}$}}
\put(200,340){\circle*{4}}\put(197,347){\makebox(0,0)[l]{\footnotesize $1_{L_1}$}}
\put(200,40){\circle*{4}}\put(198,32){\makebox(0,0)[l]{\footnotesize $0_{L_1}$}}
\put(198,12){\makebox(0,0)[l]{\footnotesize $L_1$}}
\put(500,60){\circle*{4}}\put(503,58){\makebox(0,0)[l]{\footnotesize $x_1$}}
\put(500,90){\circle*{4}}\put(502,82){\makebox(0,0)[l]{\footnotesize $x_2$}}
\put(525,115){\circle*{4}}\put(527,108){\makebox(0,0)[l]{\footnotesize $e_1$}}
\put(450,140){\circle*{4}}\put(435,140){\makebox(0,0)[l]{\footnotesize $x_3$}}
\put(450,190){\circle*{4}}\put(436,192){\makebox(0,0)[l]{\footnotesize $x_7$}}
\put(475,165){\circle*{4}}\put(468,174){\makebox(0,0)[l]{\footnotesize $x_4$}}
\put(550,140){\circle*{4}}\put(555,140){\makebox(0,0)[l]{\footnotesize $x_5$}}
\put(500,190){\circle*{4}}\put(498,198){\makebox(0,0)[l]{\footnotesize $a$}}
\put(500,240){\circle*{4}}\put(493,248){\makebox(0,0)[l]{\footnotesize $e_2$}}
\put(386,175){\circle*{4}}\put(371,177){\makebox(0,0)[l]{\footnotesize $x_8$}}
\put(575,165){\circle*{4}}\put(580,170){\makebox(0,0)[l]{\footnotesize $x_6$}}
\put(525,215){\circle*{4}}\put(530,215){\makebox(0,0)[l]{\footnotesize $x_9$}}
\put(525,265){\circle*{4}}\put(530,265){\makebox(0,0)[l]{\footnotesize $x_{10}$}}
\put(500,290){\circle*{4}}\put(504,290){\makebox(0,0)[l]{\footnotesize $x_{11}$}}
\put(500,340){\circle*{4}}\put(497,347){\makebox(0,0)[l]{\footnotesize $1_{L_2}$}}
\put(500,40){\circle*{4}}\put(498,32){\makebox(0,0)[l]{\footnotesize $0_{L_2}$}}
\put(498,12){\makebox(0,0)[l]{\footnotesize $L_2$}}

\put(201,91){\line(1,1){22}}
\put(226,116){\line(1,1){22}}
\put(199,91){\line(-1,1){48}}
\put(249,141){\line(-1,1){48}}
\put(224,114){\line(-1,1){48}}
\put(151,141){\line(1,1){22}}
\put(200,291){\line(0,1){48}}
\put(101,191){\line(1,1){98}}
\put(101,191){\line(1,-1){48}}
\put(251,141){\line(1,1){22}}
\put(274,164){\line(-1,1){48}}
\put(174,164){\line(-1,1){23}}
\put(176,166){\line(1,1){23}}
\put(201,191){\line(1,1){23}}
\put(151,191){\line(1,1){48}}
\put(224,214){\line(-1,1){22}}
\put(201,241){\line(1,1){22}}
\put(224,264){\line(-1,1){23}}
\put(200,42){\line(0,1){17}}
\put(200,62){\line(0,1){26}}

\put(501,91){\line(1,1){22}}
\put(526,116){\line(1,1){22}}
\put(499,91){\line(-1,1){48}}
\put(549,141){\line(-1,1){48}}
\put(524,114){\line(-1,1){48}}
\put(451,141){\line(1,1){22}}
\put(500,291){\line(0,1){48}}
\put(387,176){\line(1,1){112}}
\put(551,141){\line(1,1){22}}
\put(574,164){\line(-1,1){48}}
\put(474,164){\line(-1,1){23}}
\put(476,166){\line(1,1){23}}
\put(501,191){\line(1,1){23}}
\put(451,191){\line(1,1){48}}
\put(524,214){\line(-1,1){22}}
\put(501,241){\line(1,1){22}}
\put(524,264){\line(-1,1){23}}
\put(500,42){\line(0,1){17}}
\put(501,61){\line(-1,1){113}}
\put(500,62){\line(0,1){26}}
  \end{picture}
  \end{minipage}
\caption{. Two bounded lattices $L_1$ and $L_2$}
\label{FIG:eg1}
\end{figure}
\begin{table}[!h]
\centering
\setlength{\abovecaptionskip}{0pt}
\setlength{\belowcaptionskip}{0pt}
\caption{The generator $f$.}
 \begin{tabular}{c|c c c c c c c c c c c c c c c c}
  \hline
   $x$ & $0_{L_1}$ & $x_1$ & $x_2$ & $x_3$ & $x_4$ &$e_1$ &$x_5$ &$x_6$ &$a$ &$x_7$ &$x_8$ &$x_9$&$e_2$&$x_{10}$& $x_{11}$&$1_{L_1}$\\
   \hline
   $f$ & -11 & -8 & -4 & 6 & 12& $0$ &9 &14 &15 &13 &11&17&18&20 &22&24\\
   \hline
 \end{tabular}
\label{TB:1}
 \end{table}

 \begin{table}[!h]
\centering
\setlength{\abovecaptionskip}{0pt}
\setlength{\belowcaptionskip}{0pt}
\caption{The 2-uninorm $\mathbb{U}$.}
\begin{tabular}{c|c c c c c c c c c c c c c c c c}
	\hline
   $\mathbb{U}$ & $0_{L_1}$ &  $x_1$ &  $x_2$ &  $x_3$ &  $x_4$ & $e_1$ & $x_5$ & $x_6$ & $a$ &  $x_7$ & $x_8$ & $x_9$ & $e_2$ & $x_{10}$& $x_{11}$ & $1_{L_1}$\\
   \hline
            $0_{L_1}$ & $0_{L_1}$ &  $0_{L_1}$ &  $0_{L_1}$ & $0_{L_1}$ &  $0_{L_1}$&  $0_{L_1}$ & $0_{L_1}$ & $0_{L_1}$ & $0_{L_1}$& $0_{L_1}$ & $0_{L_1}$& $0_{L_1}$& $0_{L_1}$ &  $0_{L_1}$ &  $0_{L_1}$ & $0_{L_1}$\\

            $x_1$ & $0_{L_1}$ &  $0_L$ &  $0_L$ &  $x_1$ &  $x_1$&  $x_1$ & $x_1$ & $x_1$ & $x_1$ & $x_1$ & $x_1$& $x_1$ & $x_1$ &    $x_1$ & $x_1$ & $x_1$\\

            $x_2$ & $0_{L_1}$ &  $0_L$ &  $x_1$ &  $x_2$ &  $x_2$&  $x_2$ & $x_2$ & $x_2$ & $x_2$ & $x_2$ & $x_2$& $x_2$ & $x_2$ &    $x_2$ & $x_2$ & $x_2$\\

            $x_3$ & $0_{L_1}$ &  $x_1$ &  $x_2$ &  $x_4$ &    $a$&  $x_3$ &   $a$ & $a$ &   $a$ & $a$ & $a$&   $a$ &  $a$ &       $a$ &   $a$ & $a$\\

            $x_4$ & $0_{L_1}$ &  $x_1$ &  $x_2$ &  $a$ &       $a$&  $x_4$ &   $a$ & $a$ &  $a$ & $a$ & $a$&    $a$ & $a$ &        $a$ &   $a$ & $a$\\

            $e_1$ &$0_{L_1}$ &  $x_1$ &  $x_2$ &  $x_3$ &     $x_4$&  $e_1$ & $x_5$ & $a$ & $a$ & $a$ & $a$& $a$ & $a$ & $a$ & $a$ & $a$\\

            $x_5$ & $0_{L_1}$ &  $x_1$ &  $x_2$ &  $a$ &  $a$&  $x_5$ & $a$ & $a$ & $a$ & $a$ & $a$& $a$ & $a$ & $a$ & $a$ &$a$\\

            $x_6$ & $0_{L_1}$ &  $x_1$ &  $x_2$ &  $a$ &  $a$&  $a$ & $a$ & $a$ & $a$ & $a$ & $a$& $a$ & $a$ & $a$ & $a$ & $a$\\

              $a$ & $0_{L_1}$ &  $x_1$ &  $x_2$ &  $a$ &  $a$&  $a$ & $a$ & $a$ & $a$ & $a$ & $a$& $a$ & $a$ & $a$ & $a$ & $a$\\

   $x_7$ & $0_{L_1}$ &  $x_1$ &  $x_2$ &   $a$ &  $a$&  $a$ & $a$ & $a$ & $a$ & $a$ & $a$& $a$ & $a$ & $a$ & $a$ & $a$\\

   $x_8$ & $0_{L_1}$ &  $x_1$ &  $x_2$ &  $a$ &  $a$&  $a$ & $a$ & $a$ & $a$ & $a$ & $a$& $a$ & $a$ & $a$ & $a$ & $a$\\

   $x_9$ & $0_{L_1}$&  $x_1$ &  $x_2$ &  $a$  &  $a$&  $a$ & $a$ & $a$ & $a$ & $a$& $a$ & $x_9$ & $x_9$ & $x_{10}$ & $x_{11}$ & $1_{L_1}$\\

   $e_2$ & $0_{L_1}$ &  $x_1$ &  $x_2$ &  $a$  &  $a$&  $a$ & $a$ & $a$ & $a$ & $a$& $a$ & $x_9$ & $e_2$ & $x_{10}$ & $x_{11}$ & $1_{L_1}$\\

   $x_{10}$ &$0_{L_1}$&  $x_1$ &  $x_2$ &  $a$  &  $a$&  $a$ & $a$ & $a$ & $a$ & $a$& $a$ & $x_{10}$ & $x_{10}$ & $x_{10}$ & $x_{11}$ & $1_{L_1}$\\

   $x_{11}$ & $0_{L_1}$&  $x_1$ &  $x_2$ &  $a$  &  $a$&  $a$ & $a$ & $a$ & $a$ & $a$& $a$ & $x_{11}$ & $x_{11}$ & $x_{11}$ & $x_{11}$ & $1_{L_1}$\\

   $1_L$ & $0_{L_1}$ &  $x_1$ &  $x_2$ &  $a$  &  $a$&  $a$ & $a$ & $a$ & $a$ & $a$& $a$ & $1_{L_1}$ & $1_{L_1}$ &$1_{L_1}$ &$1_{L_1}$& $1_{L_1}$\\
   \hline
 \end{tabular}
 \label{TB:3}
 \end{table}
 \quad\\
\quad\\

\quad\\
\quad\\
\begin{example}
\emph{Consider the bounded lattice $(L,\leq)$ as shown in Fig. \ref{FIG:eg2}.
Let $e_1=x_4$, $e_2=x_9$ and $a=x_7$. Then the function
$f:L\to [-\infty,+\infty]$ defined by $f(x_i)=i-4$ is an injective increasing function with $f(x_4)=0$. Further, we have:
\begin{enumerate}[{\rm (i)}]
\item
If $i\in [0,3]$, then $f(x_i)\in [-4,-1]$ and we have $f(x_i)+f(x_j)\in [-\infty, f(x_0))\cup [-4,-2]$ for all $i,j\in [0,3]\cup \{4\}$.
If $i\in [5,7]$, then $f(x_i)\in [1,3]$ and we have $\min\{f(x_i)+f(x_j),f(x_7)\}\in [2,3]$ for all $i,j\in [5,7]\cup \{4\}$.
Hence the condition (i) in Theorem \ref{thm-1} is satisfied since $[-4,-2]\cup [2,3]\subseteq Ran(f)$.
\item
For all $(f(x_i),f(x_j))\in [0,f(x_7)]^2\cup [-\infty,0]^2$, if $f(x_i)$ and $f(x_j)$ have at least one same summand $f(x_k)\in Ran(f)$ then we have $i,j\in [0,2]$, or $i,j\in [6,7]$, or $i\in [0,2]$ and $j\in [2,3]$, or $i\in [5,5.5]$ and $j\in [6,7]$, subsequently, $x_i\nparallel x_j$ by Fig. \ref{FIG:eg2}. Therefore, the condition (ii) in Theorem \ref{thm-1} is satisfied.
\item
For all $f(x_i)<0$ and $0< f(x_j)\leq f(x_7)$, we have $i\in [0,3]$ and $j\in [5,7]$. Then $x_i\nparallel x_j$ by Fig. \ref{FIG:eg2} and the condition (iii) in Theorem \ref{thm-1} is satisfied.
\item
For all $f(a)\leq f(x_i)<f(x_j)$, we have $i,j\in [7,10]$. Then $x_i\nparallel x_j$ by Fig. \ref{FIG:eg2}. For all $f(x_i)<0< f(a)\leq f(x_j)$, we have $i\in [0,3]$ and $j\in [7,10]$. Thus $x_i\nparallel x_j$ by Fig. \ref{FIG:eg2}. Therefore, the condition (iv) in Theorem \ref{thm-1} is satisfied.
\item
For all $f(x_i),f(x_j)\in (0,f(x_7))$, if $f(x_k)=f(x_i)+f(x_j)\in Ran(f)$ and $0<f(x_i)+f(x_j)\leq f(x_7)$, then $f(x_k)=f(x_i)+f(x_j)\in [2,3]$. Subsequently, $k\in [6,7]$. Then the condition (v) in Theorem \ref{thm-1} is easily checked since $x_k\nparallel x_7$ by Fig. \ref{FIG:eg2}, i.e., $f^{-1}(f(x_i)+f(x_j))\nparallel x_8$.
\end{enumerate}
So that by Theorem \ref{thm-1}, the function $f$ is an additive generator of the 2-uninorm $\mathbb{U}$ given by: for all $x_i,x_j\in L$,
\begin{equation*}
\mathbb{U}(x_i,x_j)=
\begin{cases}
x_{\max\{i+j-4,0\}},& (x_i,x_j)\in ([x_0,x_4]\cup I_{x_4})^2,\\
x_{\min\{i+j-4,7\}}, & (x_i,x_j)\in [x_4,x_7]^2,\\
\min\{x_i,x_j\}, & (x_i,x_j)\in[x_7,x_9]^2\cup [x_0,x_4)\times [x_4,x_{10}]\cup [x_4,x_{10}]\times [x_0,x_4),\\
\max\{x_i,x_j\}, & (x_i,x_j)\in[x_9,x_{10}]^2\cup (x_9,x_{10}]\times (x_7,x_9]\cup (x_7,x_9]\times (x_9,x_{10}],\\
x_7, & \mbox{otherwise.}
\end{cases}
\end{equation*}}

\begin{figure}[!h]
\par\noindent\vskip50pt
 \begin{minipage}{11pc}
\setlength{\unitlength}{0.77pt}\begin{picture}(600,340)
\put(320,405){\circle*{4}}\put(317,412){\makebox(0,0)[l]{\footnotesize $x_{10}$}}
\put(320,395){\circle*{4}}\put(294,394){\makebox(0,0)[l]{\footnotesize $x_{9.9}$}}
\put(320,370){\circle*{3}}\put(317,377){\makebox(0,0)[l]{\footnotesize }}
\put(320,374){\circle*{3}}\put(317,381){\makebox(0,0)[l]{\footnotesize }}
\put(320,378){\circle*{3}}\put(317,385){\makebox(0,0)[l]{\footnotesize }}

\put(320,358){\circle*{3}}\put(317,365){\makebox(0,0)[l]{\footnotesize }}
\put(320,354){\circle*{3}}\put(317,361){\makebox(0,0)[l]{\footnotesize }}
\put(320,350){\circle*{3}}\put(317,357){\makebox(0,0)[l]{\footnotesize }}
\put(284,328){\circle*{4}}\put(259,328){\makebox(0,0)[l]{\footnotesize $x_{5.8}$}}
\put(299,342){\circle*{3}}\put(272,341){\makebox(0,0)[l]{\footnotesize }}
\put(302,345){\circle*{3}}\put(275,344){\makebox(0,0)[l]{\footnotesize }}
\put(305,348){\circle*{3}}\put(278,347){\makebox(0,0)[l]{\footnotesize }}

\put(398,317){\circle*{4}}\put(403,317){\makebox(0,0)[l]{\footnotesize $x_{5.6}$}}
\put(320,345){\circle*{4}}\put(324,347){\makebox(0,0)[l]{\footnotesize $x_7$}}
\put(320,327){\circle*{4}}\put(324,329){\makebox(0,0)[l]{\footnotesize $x_6$}}
\put(320,332){\circle*{3}}\put(324,334){\makebox(0,0)[l]{\footnotesize }}
\put(320,336){\circle*{3}}\put(324,338){\makebox(0,0)[l]{\footnotesize }}
\put(320,340){\circle*{3}}\put(324,342){\makebox(0,0)[l]{\footnotesize }}
\put(320,364){\circle*{4}}\put(324,366){\makebox(0,0)[l]{\footnotesize $x_9$}}
\put(314,358){\circle*{4}}\put(290,360){\makebox(0,0)[l]{\footnotesize $x_{5.9}$}}

\put(299,291){\circle*{3}}\put(269,291){\makebox(0,0)[l]{\footnotesize }}
\put(302,294){\circle*{3}}\put(272,294){\makebox(0,0)[l]{\footnotesize }}
\put(305,297){\circle*{3}}\put(275,297){\makebox(0,0)[l]{\footnotesize }}
\put(284,276){\circle*{4}}\put(260,276){\makebox(0,0)[l]{\footnotesize $x_{5.7}$}}
\put(287,273){\circle*{3}}\put(263,273){\makebox(0,0)[l]{\footnotesize}}
\put(290,270){\circle*{3}}\put(266,270){\makebox(0,0)[l]{\footnotesize}}
\put(293,267){\circle*{3}}\put(269,267){\makebox(0,0)[l]{\footnotesize}}

\put(320,55){\circle*{3}}\put(323,53){\makebox(0,0)[l]{\footnotesize }}
\put(320,50){\circle*{3}}\put(323,48){\makebox(0,0)[l]{\footnotesize }}
\put(320,45){\circle*{3}}\put(323,43){\makebox(0,0)[l]{\footnotesize }}
\put(320,90){\circle*{4}}\put(322,82){\makebox(0,0)[l]{\footnotesize $x_{2}$}}
\put(320,120){\circle*{4}}\put(324,122){\makebox(0,0)[l]{\footnotesize $x_{2.1}$}}
\put(320,160){\circle*{4}}\put(323,162){\makebox(0,0)[l]{\footnotesize $x_{2.2}$}}
\put(331,101){\circle*{4}}\put(334,98){\makebox(0,0)[l]{\footnotesize $x_{2.5}$}}
\put(309,101){\circle*{4}}\put(287,98){\makebox(0,0)[l]{\footnotesize $x_{2.3}$}}
\put(320,135){\circle*{3}}\put(318,127){\makebox(0,0)[l]{\footnotesize }}
\put(320,139){\circle*{3}}\put(318,131){\makebox(0,0)[l]{\footnotesize }}
\put(320,143){\circle*{3}}\put(318,134){\makebox(0,0)[l]{\footnotesize }}

\put(340,110){\circle*{3}}\put(342,103){\makebox(0,0)[l]{\footnotesize }}
\put(343,113){\circle*{3}}\put(345,106){\makebox(0,0)[l]{\footnotesize }}
\put(346,116){\circle*{3}}\put(348,109){\makebox(0,0)[l]{\footnotesize }}
\put(270,140){\circle*{4}}\put(247,140){\makebox(0,0)[l]{\footnotesize $x_{2.4}$}}
\put(297,113){\circle*{3}}\put(297,106){\makebox(0,0)[l]{\footnotesize }}
\put(293,116){\circle*{3}}\put(293,109){\makebox(0,0)[l]{\footnotesize }}
\put(290,119){\circle*{3}}\put(290,112){\makebox(0,0)[l]{\footnotesize }}
\put(370,140){\circle*{4}}\put(375,140){\makebox(0,0)[l]{\footnotesize $x_{2.71}$}}
\put(350,161){\circle*{3}}\put(343,170){\makebox(0,0)[l]{\footnotesize }}
\put(346,165){\circle*{3}}\put(339,174){\makebox(0,0)[l]{\footnotesize }}
\put(341,169){\circle*{3}}\put(334,178){\makebox(0,0)[l]{\footnotesize }}
\put(320,190){\circle*{4}}\put(322,195){\makebox(0,0)[l]{\footnotesize $x_3$}}
\put(330,180){\circle*{4}}\put(334,182){\makebox(0,0)[l]{\footnotesize $x_{2.79}$}}
\put(310,180){\circle*{4}}\put(283,182){\makebox(0,0)[l]{\footnotesize $x_{2.49}$}}
\put(320,212){\circle*{4}}\put(323,215){\makebox(0,0)[l]{\footnotesize $x_4$}}
\put(291,161){\circle*{3}}\put(284,170){\makebox(0,0)[l]{\footnotesize }}
\put(295,165){\circle*{3}}\put(288,174){\makebox(0,0)[l]{\footnotesize }}
\put(299,169){\circle*{3}}\put(292,178){\makebox(0,0)[l]{\footnotesize }}
\put(320,240){\circle*{4}}\put(325,240){\makebox(0,0)[l]{\footnotesize $x_{5}$}}
\put(320,265){\circle*{4}}\put(323,268){\makebox(0,0)[l]{\footnotesize $x_{5.1}$}}
\put(320,300){\circle*{4}}\put(324,300){\makebox(0,0)[l]{\footnotesize $x_{5.49}$}}
\put(320,313){\circle*{4}}\put(324,313){\makebox(0,0)[l]{\footnotesize $x_{5.79}$}}
\put(335,255){\circle*{4}}\put(340,251){\makebox(0,0)[l]{\footnotesize $x_{5.51}$}}
\put(305,255){\circle*{4}}\put(279,250){\makebox(0,0)[l]{\footnotesize $x_{5.61}$}}
\put(358,278){\circle*{3}}\put(363,258){\makebox(0,0)[l]{\footnotesize }}
\put(362,282){\circle*{3}}\put(367,262){\makebox(0,0)[l]{\footnotesize }}
\put(366,286){\circle*{3}}\put(371,266){\makebox(0,0)[l]{\footnotesize }}

\put(300,220){\circle*{4}}\put(272,220){\makebox(0,0)[l]{\footnotesize $x_{2.99}$}}
\put(340,220){\circle*{4}}\put(344,220){\makebox(0,0)[l]{\footnotesize $x_{2.9}$}}
\put(320,280){\circle*{3}}\put(325,280){\makebox(0,0)[l]{\footnotesize }}
\put(320,284){\circle*{3}}\put(325,284){\makebox(0,0)[l]{\footnotesize }}
\put(320,288){\circle*{3}}\put(325,288){\makebox(0,0)[l]{\footnotesize }}

\put(278,200){\circle*{3}}\put(264,200){\makebox(0,0)[l]{\footnotesize }}
\put(282,204){\circle*{3}}\put(268,204){\makebox(0,0)[l]{\footnotesize }}
\put(286,208){\circle*{3}}\put(272,208){\makebox(0,0)[l]{\footnotesize }}
\put(245,165){\circle*{4}}\put(219,170){\makebox(0,0)[l]{\footnotesize $x_{2.91}$}}
\put(395,165){\circle*{4}}\put(399,169){\makebox(0,0)[l]{\footnotesize $x_{2.8}$}}
\put(363,198){\circle*{3}}\put(368,198){\makebox(0,0)[l]{\footnotesize }}
\put(360,201){\circle*{3}}\put(365,201){\makebox(0,0)[l]{\footnotesize }}
\put(357,204){\circle*{3}}\put(362,204){\makebox(0,0)[l]{\footnotesize }}
\put(320,30){\circle*{4}}\put(325,28){\makebox(0,0)[l]{\footnotesize $x_{0.1}$}}
\put(320,20){\circle*{4}}\put(318,12){\makebox(0,0)[l]{\footnotesize $x_0$}}

\put(320,240){\line(-1,1){24}}
\put(320,364){\line(-1,-1){13}}
\put(284,328){\line(1,1){13}}
\put(284,276){\line(0,1){50}}
\put(320,313){\line(-1,-1){13}}
\put(284,276){\line(1,1){13}}
\put(398,317){\line(-1,1){78}}
\put(398,317){\line(-1,-1){30}}
\put(320,405){\line(0,-1){22}}
\put(320,328){\line(0,-1){34}}
\put(320,240){\line(1,1){35}}
\put(320,240){\line(0,1){34}}
\put(320,240){\line(-1,-1){30}}
\put(320,240){\line(1,-1){33}}
\put(320,240){\line(0,-1){50}}
\put(245,165){\line(1,1){30}}
\put(395,165){\line(-1,1){30}}
\put(370,140){\line(1,1){25}}
\put(270,140){\line(-1,1){25}}
\put(320,190){\line(1,-1){17}}
\put(320,190){\line(-1,-1){17}}
\put(270,140){\line(1,1){17}}
\put(270,140){\line(1,-1){17}}
\put(320,190){\line(0,-1){40}}
\put(370,140){\line(-1,1){17}}
\put(370,140){\line(-1,-1){17}}
\put(320,90){\line(0,1){40}}
\put(320,90){\line(1,1){15}}
\put(320,90){\line(-1,1){17}}
\put(320,92){\line(0,-1){30}}
\put(320,22){\line(0,1){17}}

  \end{picture}
  \end{minipage}
\caption{. A bounded lattice $L$}
\label{FIG:eg2}
\end{figure}
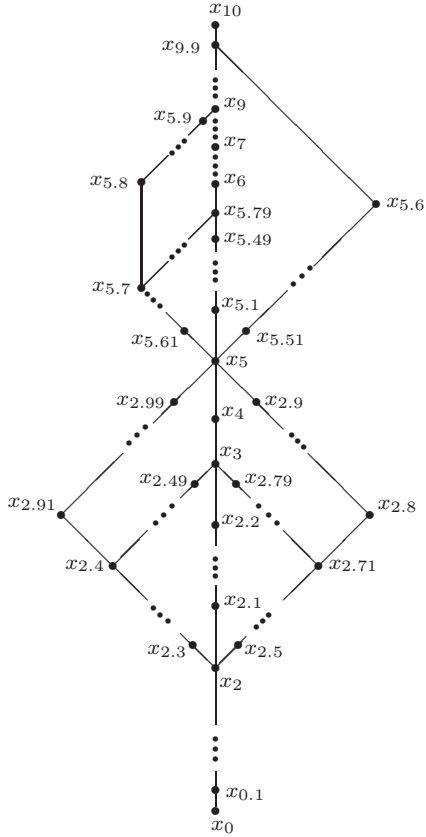
\end{example}

The following example shows that each of the conditions (i) - (v) in Theorems \ref{thm-1} can't be dropped, respectively.

\begin{example}
\emph{
Consider the two bounded lattices $L_1$ and $L_2$ in Fig. \ref{FIG:eg1} and the five injective increasing functions $f_i$, $i= 1,2,3,4,5$, with $f_i(e_1)=0$ defined by Tables \ref{TB:5}, \ref{TB:6}, \ref{TB:7}, \ref{TB:8} and \ref{TB:9}, respectively.
\begin{table}[!h]
\centering
\setlength{\abovecaptionskip}{0pt}
\setlength{\belowcaptionskip}{0pt}
\caption{The generator $f_1$.}
\label{TB:5}
 \begin{tabular}{c|c c c c c c c c c c c c c c c c}
  \hline
   $x$ & $0_{L_1}$ & $x_1$ & $x_2$ & $x_3$ & $x_4$ &$e_1$ &$x_5$ &$x_6$ &$a$ &$x_7$ &$x_8$ &$x_9$&$e_2$&$x_{10}$& $x_{11}$&$1_{L_1}$\\
   \hline
   $f_{1}$ & -11 & -8 & -4 & 6 & 10& $0$ &9 &14 &15 &13 &11&17&18&20 &22&24\\
   \hline
 \end{tabular}
\end{table}
\begin{table}[!h]
\centering
\setlength{\abovecaptionskip}{0pt}
\setlength{\belowcaptionskip}{0pt}
 \caption{The generator $f_2$.}
 \label{TB:6}
 \begin{tabular}{c|c c c c c c c c c c c c c c c c}
  \hline
   $x$ & $0_{L_1}$ & $x_1$ & $x_2$ & $x_3$ & $x_4$ &$e_1$ &$x_5$ &$x_6$ &$a$ &$x_7$ &$x_8$ &$x_9$&$e_2$&$x_{10}$& $x_{11}$&$1_{L_1}$\\
   \hline
   $f_{2}$ & -11 & -8 & -4 & 6 & 9& $0$ &12 &14 &15 &13 &11&17&18&20 &22&24\\
   \hline
 \end{tabular}
 \end{table}
 \begin{table}[!h]
\centering
\setlength{\abovecaptionskip}{0pt}
\setlength{\belowcaptionskip}{0pt}
\caption{The generator $f_3$.}
 \begin{tabular}{c|c c c c c c c c c c c c c c c c}
  \hline
   $x$ & $0_{L_2}$ & $x_1$ & $x_2$ & $x_3$ & $x_4$ &$e_1$ &$x_5$ &$x_6$ &$a$ &$x_7$ &$x_8$ &$x_9$&$e_2$&$x_{10}$& $x_{11}$&$1_{L_2}$\\
   \hline
   $f_{3}$ & -11 & -8 & -4 & 6 & 12& $0$ &9 &14 &15 &13 &11&17&18&20 &22&24\\
   \hline
 \end{tabular}
  \label{TB:7}
 \end{table}
  \begin{table}[!h]
\centering
\setlength{\abovecaptionskip}{0pt}
\setlength{\belowcaptionskip}{0pt}
\caption{The generator $f_4$.}
 \begin{tabular}{c|c c c c c c c c c c c c c c c c}
  \hline
   $x$ & $0_{L_1}$ & $x_1$ & $x_2$ & $x_3$ & $x_4$ &$e_1$ &$x_5$ &$x_6$ &$a$ &$x_7$ &$x_8$ &$x_9$&$e_2$&$x_{10}$& $x_{11}$&$1_{L_1}$\\
   \hline
   $f_{4}$ & -11 & -8 & -4 & 6 & 12& $0$ &9 &14 &15 &13 &19&17&18&20 &22&24\\
   \hline
 \end{tabular}
  \label{TB:8}
 \end{table}
 \begin{table}[!h]
\centering
\setlength{\abovecaptionskip}{0pt}
\setlength{\belowcaptionskip}{0pt}
\caption{The generator $f_5$.}
 \begin{tabular}{c|c c c c c c c c c c c c c c c c}
  \hline
   $x$ & $0_{L_1}$ & $x_1$ & $x_2$ & $x_3$ & $x_4$ &$e_1$ &$x_5$ &$x_6$ &$a$ &$x_7$ &$x_8$ &$x_9$&$e_2$&$x_{10}$& $x_{11}$&$1_{L_1}$\\
   \hline
   $f_{5}$ & -11 & -8 & -4 & 9 & 13& $0$ &6 &12 &15 &14 &11&17&18&20 &22&24\\
   \hline
 \end{tabular}
  \label{TB:9}
 \end{table}
\begin{enumerate}[{\rm (i)}]
  \item
One can check that $f_1: L_1\to [-11, 24]$ satisfies (ii), (iii), (iv) and (v), but it doesn't satisfy (i) in Theorem \ref{thm-1} since $f_1(x_3)+f_1(x_3)=12\notin Ran(f_1)\cup [-\infty,f_1(0_L)]$. Applying Formula \eqref{eq-1}, we know that $\mathbb{U}(\mathbb{U}(x_3,x_3),x_4)=\mathbb{U}(e_1,x_4)=x_4\neq a=\mathbb{U}(x_3,a)=\mathbb{U}(x_3,\mathbb{U}(x_3,x_4))$.
Thus $\mathbb{U}$ isn't a 2-uninorm.
  \item
Also, one can easily verify that $f_2: L_1\to [-11, 24]$ satisfies (i), (iii), (iv) and (v), but it doesn't satisfy (ii) of Theorem \ref{thm-1} since both $f_2(x_5)$ and $f_2(x_3)$ have a same summand $f_2(x_3)$ but $x_3\parallel x_5$. By using Formula \eqref{eq-1}, we have
$\mathbb{U}(x_3,e_1)=x_3\parallel x_5=\mathbb{U}(x_3,x_3)$ while $e_1\leq x_3$. It follows that $\mathbb{U}$ isn't a 2-uninorm.
  \item
One can clarify that $f_3: L_2\to [-11,24]$ satisfies (i),(ii),(iv) and (v), but $f_3$ doesn't satisfy (iii) in Theorem \ref{thm-1} since $f_3(x_2)<0$, $x_2<e_1$ but $x_2 \parallel x_8$. By using Formula \eqref{eq-1}, we obtain $\mathbb{U}(x_2,x_8)=x_2$ and
$\mathbb{U}(e_1,x_8)=x_8$. Thus $\mathbb{U}(x_2,x_8) \parallel \mathbb{U}(e_1,x_8)$ while $x_2< e_1$. It follows that $\mathbb{U}$ isn't a
2-uninorm.
  \item
One can prove that $f_4: L_1\to [-11,24]$ satisfies (i), (ii), (iii) and (v), but $f_3$ doesn't satisfy (iv) in Theorem \ref{thm-1} since $f_4(a)<f_4(x_{9})<f_4(x_{10})$, but $x_{8} \parallel x_{10}$. By using Formula \eqref{eq-1}, one has $\mathbb{U}(x_{9},x_{8})=x_{8}$ and $\mathbb{U}(x_{10},x_{8})=x_{10}$. Thus $\mathbb{U}(x_{9},x_{8})\parallel \mathbb{U}(x_{10},x_{8})$ while $x_{9}<x_{10}$. It follows that $\mathbb{U}$ isn't a 2-uninorm.
  \item
One can check that $f_5: L_1\to [-11,24]$ satisfies (i), (ii), (iii) and (iv), but $f_5$ doesn't satisfy (v) in Theorem \ref{thm-1} since $0<f_5(x_5)<f_5(a)$, but $x_{6} \nparallel a$. By using Formula \eqref{eq-1}, one knows $\mathbb{U}(x_5,x_5)=x_{6}$ and $\mathbb{U}(x_5,x_6)=a$. Thus $\mathbb{U}(x_5,x_5)=f_5^{-1}(f_5(x_5)+f_5(x_5))=x_6 \parallel a=\mathbb{U}(x_{5},x_{6})$ while $x_5<x_6$. It follows that $\mathbb{U}$ isn't a
2-uninorm.
\end{enumerate}
}
\end{example}

\section{Conclusions}\label{sec6}
In this article, we mainly presented two construction methods of 2-uninorms on bounded lattices by using additive generators, and also supplied two examples to show that the existence of such additive generators on bounded lattices. It is
worth pointing out that uninorms, nullnorms, uni-nullnorms, null-uninorms are all special 2-uninorms, respectively. Consequently,
we can obtain their additive generators. Dually, one may discuss the multiplicative generators of these binary operations. It is an interesting problem to further find some lax conditions in both Theorems \ref{thm-1} and \ref{thm-2} by modifying Formulas \eqref{eq-1} and \eqref{eq-2}, respectively.

\end{document}